\documentclass[reqno,twoside,a4paper,11pt]{amsart}
\usepackage{amsmath,amssymb,dsfont,amsthm}
\usepackage{graphics}
\usepackage{xypic}
\usepackage{latexsym}
\usepackage{amsfonts}
\usepackage{geometry}
\usepackage{tensor}
\usepackage{bm,color}
\usepackage{relsize}
\usepackage{hyperref}

\usepackage{upgreek}

\hypersetup{
    colorlinks=true,
    linkcolor=blue,
    filecolor=magenta,      
    urlcolor=cyan,
    citecolor=blue,
    pdftitle={Sharelatex Example},
    bookmarks=true,
    pdfpagemode=FullScreen,
}

\date{\today}

\title[]{Localised analytic torsion and relative analytic torsion for non compact Lie groups of type I}

\author{A. Della Vedova}
\address[Mauro Spreafico]{\tt Dipartimento di matematica ed applicazioni, Universit\`a Milano Bicocca,, Italy}
\email{alberto.dellavedova@unimib.it}

\author{M. Spreafico}
\address[Mauro Spreafico]{\tt Dipartimento di matematica ed applicazioni, Universit\`a Milano Bicocca,  and INFN Lecce, Italy}
\email{mauro.spreafico@unimib.it}

\numberwithin{equation}{section}

\newtheorem{lem}[subsubsection]{Lemma}
\newtheorem{corol}[subsubsection]{Corollary}

\newtheorem{prop}[subsubsection]{Proposition}

\newtheorem{ass}[subsubsection]{Assumption}

\newcommand{\beq}{\begin{equation}}
\newcommand{\eeq}{\end{equation}}

\newcommand{\te}{\theta}
\newcommand{\de}{{\delta}}
\newcommand{\vv}{{\varphi}}
\newcommand{\ep}{\epsilon}
\newcommand{\al}{\alpha}
\newcommand{\be}{\beta}

\newcommand{\ka}{\kappa}
\newcommand{\la}{\lambda}
\newcommand{\ga}{\gamma}

\newcommand{\om}{\omega}
\renewcommand{\ep}{\epsilon}

\newcommand{\ze}{\zeta}
\newcommand{\bze}{\bar\zeta}

\renewcommand{\Re}{{\rm Re}}

\renewcommand{\b}{{\partial}}
\newcommand{\Vol}{{\rm Vol}}

\newcommand{\da}{{\dagger}}

\newcommand{\bu}{{\bullet}}
\newcommand{\Sp}{{\rm Sp}}
\newcommand{\sgn}{{\rm sgn}}
\renewcommand{\d}{{\rm d}}

\newcommand{\Tr}{{\rm Tr\hspace{.5pt} }}
\newcommand{\tr}{{\rm tr\hspace{.5pt} }}

\newcommand{\lp}{\langle}
\newcommand{\rp}{\rangle}
\newcommand{\bz}{{\bar{z}}}
\newcommand{\BZ}{{\bar{Z}}}
\newcommand{\bsl}{{\backslash}}

\newcommand{\N}{{\mathds{N}}}
\newcommand{\Z}{{\mathds{Z}}}
\newcommand{\R}{{\mathds{R}}}
\newcommand{\C}{{\mathds{C}}}

\newcommand{\T}{{\mathds{T}}}
\renewcommand{\H}{{\mathds{H}}}

\newcommand\e{{\rm e}}

\newcommand{\hs}{\mathsf{h}}

\newcommand{\hf}{\mathfrak{h}}
\newcommand{\gf}{{\mathfrak{g}}}

\newcommand{\tf}{\mathfrak{t}}

\newcommand{\Rho}{\mathrm{P}}

\newcommand{\LL}{\mathcal{L}}

\renewcommand{\H}{\mathcal{H}}

\newcommand{\FF}{\mathcal{F}}

\newcommand{\CC}{\mathcal{C}}
\renewcommand{\SS}{\mathcal{S}}

\newcommand{\TF}{\mathfrak{T}}

\newcommand{\DS}{\mathsf{D}}

\newcommand{\B}{{\mathcal{B}}}
\newcommand{\HH}{{\mathcal{H}}}

\newcommand{\A}{{\mathcal{A}}}

\newcommand{\UU}{\mathcal{U}}
\newcommand{\TT}{\mathcal{T}}

\newcommand{\RR}{\mathcal{R}}

\DeclareMathOperator*{\Rz}{Res_0}
\DeclareMathOperator*{\Ru}{Res_1}

\begin{document}

\maketitle
\begin{abstract} Let $G$ be  a (non compact) connected simply connected locally compact second countable Lie group, either abelian or unimodular  of type I, and $\rho$ an irreducible unitary representation of $G$. Then, we define the analytic torsion of $G$ localised at the representation $\rho$. Next, let $\Gamma$ a discrete cocompact subgroup of $G$. We use the localised analytic torsion to define the relative analytic torsion of the pair $(G,\Gamma)$, and we prove that it coincides with the Lott $L^2$ analytic torsion of a covering space. We illustrate these constructions analysing in some details two examples: the abelian case,  and the case  $G=H$,  the Heisenberg group. 
\end{abstract}

\tableofcontents 

\section{Introduction}

The aim of this notes is to discuss some possible generalisation of analytic torsion \cite{RS} to non compact Lie groups. For we first  introduce and investigate the concept of localised analytic torsion for a non compact Lie group $G$. In few words, localised analytic torsion is the analytic torsion of the field of operators defined by localising  the natural Hodge Laplace operator on $G$  at some representation of $G$. This is the natural counterpart of the localised eta function introduced in \cite{Spr2}. Second, assuming that $G$ has a discrete cocompact subgroup $\Gamma$, we define what we call the relative analytic torsion of the pair $(G,\Gamma)$. This is the natural analogue  of the $L^2$ analytic torsion of a covering space introduced by J. Lott \cite{Lot}, and indeed we prove their equivalence. In order to illustrate these constructions, we present a detailed analysis of two particular examples: the abelian case, $G=\R$, and the case of $G=H$, the Heisenberg group. We conclude by  propose a "geometric" interpretation of relative analytic torsion in terms of the classical analytic torsion. This works nicely in the abelian case, while it is less natural in the case of the Heisenberg group.

Let now see in more details the plan of the work. 
The Reidemeister torsion (R torsion) of a finite combinatorial complex $K$ is a kind of determinant in the Whitehead group $Wh(\Z\pi_1(K))$, which describes the way in which the 'cells' of the universal covering complex are fitted together with respect to the action of the fundamental group \cite{Mil}. When $K$ is a triangulation of a compact connected oriented Riemannian manifold $(M,g)$ (for simplicity without boundary), it is natural to 'change' the ring by some orthogonal representation $\al:\pi_1(M)\to O(V)$ of the fundamental group. Then, R torsion is a topological invariant of the triple $(M,\hs,\al)$, where $\hs$ is a basis for the rational homology determinant line bundle of $M$. In this situation, Ray and Singer introduced an analytic object called analytic torsion, constructed from the de Rham complex of $M$ twisted by $\al$, and conjectured its equivalence with R torsion, when the homology graded basis $\hs$ is the one fixed by an orthonormal graded basis of harmonic forms determined by $g$ \cite{RS}. J. Cheeger \cite{Che} and W. M\"{u}ller \cite{Mul} proved independently this conjecture. The analytic torsion is defined as follows. Let $(\Omega^{\bu}(M, E_\al),d^{\bu})$ be the de Rham complex of the forms on $M$ with coefficients in the real vector bundle $E_\al=\widetilde M \times_\al V$, associated to $\al$. The Riemannian metric on $M$ provides an inner product and makes the de Rham complex a complex of Hilbert spaces. Let $\Delta_\al^\bu$ be the associated Hodge Laplace operator. Whence, $\Delta_\al^\bu$ will have a pure point spectrum, with unique accumulation point at infinity. Then, for complex $s$, with $\Re(s)$ large, we define the zeta function of $\Delta_\al^q$ by \cite{RS} 
\[
\zeta(s,\Delta_\al^q)=\sum_{0\not= \la\in \Sp \Delta_\al^{(q)}} \la^{-s}.
\]

Setting (that we call analytic torsion zeta function)
\[
t(s;(M,g),\al)=\frac{1}{2}\sum_{q=0}^{\dim M} (-1)^q q \zeta(s,\Delta_\al^{(q)}),
\]
we define the analytic torsion of $M$ in the representation $\al$: 
\[
T((M,g),\al)=t'(0;(M,g),\al).
\]

If $M$ is not compact, it is not clear what torsion would be. In \cite{Spr2}, we proposed un analogue of the eta function (that is a spectral invariant defined in terms of some zeta function for compact manifold, strictly similar to the analytic torsion), called localised eta function for non compact Lie groups. Here we consider the corresponding definition for analytic torsion. For let $G$ be  a (non compact) connected simply connected locally compact second countable Lie group, either abelian or unimodular  of type I, provided with a a fixed   Haar measure. Then $G$ is a smooth Riemannian manifold with a particular metric determined by left invariance. For this reason we will omit explicit reference to the metric in the notation in what follows.  Let $\Delta$ denote (some self adjoint extension of the formal)  Hodge Laplace operator on square integrable forms on $G$ associated to this metric.  Let $\rho$ be some irreducible unitary representation of $G$. There exists a measurable field of Hilbert spaces  $h\mapsto \H_h$, called the canonical field, and a field of measurable representations $h\mapsto \rho_h$, such that $\rho_h$ belongs to the class $h\in \hat G$, the dual group of $G$ \cite[8.6.1, 8.6.2]{Dix2}. In this setting, $\Delta$ determines a continuous field of self adjoint operators, $d\rho_h \Delta$. Assuming that for each fixed $h$ the spectrum $\Sp d\rho_h \Delta$ of $d\rho_h \Delta$ is discrete sequence with some suitable properties (such a sequence are called of spectral type in \cite{Spr1}, see also \cite{Vor}), we may associate to $\Sp d\rho_h \Delta$ some spectral  functions, and in particular the zeta function. If one is able to prove that the analytic extension of this zeta function is regular at the origin, one can  define the localised analytic torsion of $d\rho_h \Delta$ as the derivative at the origin of the zeta function. All this is in Section \ref{lt}.

Next, assuming that $G$ has a  discrete cocompact subgroup $\Gamma$, we may introduce some analogue of analytic torsion for the pair $(G,\Gamma)$. The construction is inspired by the one proposed by Lott in \cite{Lot} (see also \cite{Mat} and \cite{CM}), for a covering space.  This was largely based on the properties of the heat kernel. In fact, using the Mellin transform,  the zeta function of the Hodge Laplace operator  may be restated by
\[
\zeta(s,\Delta_\al^q)=\frac{1}{\Gamma(s)}\int_0^\infty t^{s-1} \Tr \e^{-t \Delta_\al^{(q)}} dt,
\]
and the knowledge of the behaviour of the trace of the heat operator for large and small $t$ are the main tools to study the analytic properties of the zeta function.  Whence, a key point in the construction of the analogue of the analytic torsion for covering is to replace the trace of the heat operator with a suitable alternative. This alternative is given by the Diximier trace. Our definition of what we call relative torsion of the pair $(G,\Gamma)$ follows the alternative approach of exploiting the localised analytic torsion and gluing it along the dual group. For we need some technical result about traces and fields of operator that we develop in Section \ref{T}. In the second part of Section \ref{lt} is our definition of relative torsion for a pari $(G,\Gamma)$. 

Eventually, we consider two examples, where we illustrate our results. The first is the abelian case, developed in Section \ref{ab}. In the last Section \ref{H}, we study the case $G=H$ the three dimensional Heisenberg group. Working out the abelian case, it turns out that the relative analytic torsion of the pair $(\R,\Z)$ has a natural description as the continuous sum in the Plancharel measure of the analytic torsion of the quotient space $\T=\R/\Z$ over the irreducible representation of $\Z$. We try to find out a similar interpretation for the relative torsion of the pair $(H,\Gamma)$ (where $\Gamma$ is the discrete integral Heisenberg group). Unfortunately, we succeed partially: namely the spectral invariant appearing are not so natural as in the abelian case.

\section{Traces and fields of operators}
\label{T}

In this work we assume that $G$ is a connected simply connected locally compact second countable Lie group, either abelian or unimodular  of type I, with  a discrete cocompact subgroup $\Gamma$, and a fixed   Haar measure $dg$. These assumptions may appear quite overwhelming, however they cover the cases of interest in this work. A comprehensive approach for a larger family of Lie groups seems extremely unlikely, but other families  may be tackled with the suitable technical arrangements. 

Note that in the following we will use results of \cite{Dix2}, given there for separable postliminal groups. However, for second countable groups type I and postliminal are equivalent \cite[13.9.4]{Dix2}. 

Let $\hat G$ be denote the dual space of equivalence classes of unitary representations  of $G$. There exists a measurable field of Hilbert spaces  $h\mapsto \H_h$, called the canonical field, and a field of measurable representations $h\mapsto \rho_h$, such that $\rho_h$ belongs to the class $h\in \hat G$ \cite[8.6.1, 8.6.2]{Dix2}. Then, the Plancharel theorem states that there exists a unique positive measure $\d\mu(h)$ on $\hat G$ (the Plancharel measure), and an   isomorphisms \cite[18.8.1,2,3]{Dix2} 
\[
\FF:L^2(G)\to \int_{\hat G}^\oplus \H_h\otimes \bar \H_h d\mu(h),
\]
that extends to the following isomorphisms (all denoted by the same symbol)
\begin{align*}
\FF&:L\to \int_{\hat G}^\oplus \rho_h\otimes 1 d\mu(h),&
\FF&:R\to \int_{\hat G}^\oplus 1\otimes \bar \rho_h d\mu(h),\\
\FF&:\LL\to \int_{\hat G}^\oplus B(\H_h)\otimes \C d\mu(h),&
\FF&:\RR\to \int_{\hat G}^\oplus \C\otimes  B(\bar\H_h) d\mu(h),
\end{align*}
where $\LL$ and $\RR$ are the von Neumann algebras on $L^2(G)$ generated respectively by the left and right regular representations $L$ and $R$ of $G$, and $B(\H_h)$ denotes the space of of the bounded operators on $\H_h$. 

In particular, if $u\in L^1(G)\cap L^2(G)$, 
\[
\hat u(h)=\FF(u)(h)=\int_G u(g) \rho_h(g^{-1}) dg,
\]
the group Fourier transform, and  we have the inversion formula \cite[(7.38)]{Fol1}
\[
u(g)=\FF^{-1} \FF (u)(g)=\int_{\hat G} \Tr \left(\rho_h (g) \FF(u)(h)\right) d\mu(h).
\]

The last point, that is the more relevant for our analysis, is the isomorphism
\[
\FF: \Tr_G\to \int_{\hat G}^\oplus \Tr d\mu (h),
\]
where $\Tr_G$ is the natural Diximier trace on $L^2(G)$ (we described it in more details in the following), and $\Tr$ the standard trace of a trace class operator in $(B(\H_h)\otimes \C)^+$ (the subspace of positive operators). 

To proceed, take a positive  operator $T$ in $\LL$ (or in $\RR$). Then, we have a measurable field of operators $h\to \FF A\FF^{-1}(h)$ on $\int_{\hat G}^\oplus \H_h\otimes \C d\mu(h)$ (or $\int_{\hat G}^\oplus \C\otimes \bar\H_h d\mu(h)$), and 
\beq\label{ee1}
\Tr_G T= \int_{\hat G}^\oplus \Tr \FF T\FF^{-1}(h) d\mu (h).
\eeq

We explore this construction in some details, specialising to the family of operators we will be interested in. 

\subsubsection{The Diximier trace} We detail the definition of the trace $\Tr_G$ and we provide a local description, passing through the introduction of a $\Gamma$ trace, main reference for this part are \cite{Ati} \cite{AS} \cite{Dix1}. To begin with,  take $f\in L^1(G)\cap L^2(G)$. By the Plancharel Theorem,
\[
f(e)=\int_{\hat G}^\oplus \Tr \hat f(h) d\mu(h),
\]
is finite since the operator $\hat f(h)$ is Hilbert Schmidt, and therefore defines a trace on $L^1(G)\cap L^2(G)$. Since, $C_0^\infty(G)\subseteq L^1(G)\cap L^2(G)$, we have that $C_0^\infty(G)\subseteq\lp R(L^1(G)\cap L^2(G))\rp$ as subalgebras (the first with convolution product), and the inclusion is dense, and hence we may extend   by closure and linearity, to have a trace on the von Neumann algebra $\lp R(L^1(G)\cap L^2(G))\rp$,   
that, viewed as a subalgebra of $R(L^1(G))$, coincides with the subalgebra of the von Neumann algebra generated by $R(G)$, i.e. $\RR_G$. Therefore, we have a function with real values, and restricting to the family $\RR_G^+$ of the operators such that this value is positive, we have a trace:
\begin{align*}
\Tr_G&:\RR_G^+ \to \R^+,\\
\Tr_G&:A\mapsto \Tr_G A.
\end{align*}

In particular, if $A=R(f)\in \RR_G^+$,
\[
R(f)(u)(h)=\int_G f(g)R(g)(u)(h) dg=\int_G f(g)u(hg ) dg=u\star f (h),
\]
then, 
\[
\Tr_G R(f)=f(e).
\]

Next,  since restriction of left multiplication defines a unitary action $\ell$ of $\Gamma$  on $L^2(G)$, we have the identification
\[
L^2(G)\cong l^2(\Gamma)\otimes L^2(Q)\cong l^2(\Gamma)\otimes L^2(\Gamma\bsl G),
\]
where $Q$ is a fundamental domain of $\ell$. 
Denote by $\B$ the von Neumann algebra of the bounded operators on $L^2(G)$ that commute with $\Gamma$, 
$\B=\{T\in B(L^2( G))~|~T\ell(\ga)=\ell(\ga)T, \forall \ga\in \Gamma\}$.



Under the identification above, the action of $\Gamma$ on $L^2(G)$ corresponds to the left regular representation $L_\Gamma$ of $\Gamma$ on $l^2(\Gamma)$ extended by the identity on $L^2(\Gamma\bsl G)$: $\ell\cong  L_\Gamma\otimes 1$. But then, the von Neumann algebra $\B$ is the von Neumann algebra generated by $R_\Gamma(\Gamma)$ tensor the space of all bounded operators,
\[
\B\cong \RR_\Gamma\otimes B(L^2(\Gamma\bsl G)).
\]

We may now define a trace on $\B$ as follows. A trace on $\RR_\Gamma$ is defined on the generators, i.e. on $R(\Gamma)$, by
\begin{align*}
\tr_\Gamma&:R(\Gamma)\to \R^+,\\
\tr_\Gamma&:R(\gamma)\mapsto \delta_{\gamma,e}.
\end{align*}
and extended by closure and linearity. A trace on $B(L^2(\Gamma\bsl G))$ is the Hilbert  Schmidt trace $\Tr$, so we put
\begin{align*}
\Tr_\Gamma&:\RR_\Gamma\otimes B(L^2(\Gamma\bsl G))\to \R^+,\\
\Tr_\Gamma&:S\otimes T\mapsto \tr_\Gamma S ~\Tr T,
\end{align*}
and we say that $S\otimes T$ is of $\Gamma$-trace class if its trace is finite. 

There is a local description, more suitable to work in the smooth category,  of this trace as follows \cite[pg. 58 and Proposition 4.16, pg. 63]{Ati}. Any bounded operator  $A$ on $L^2(G)$ has a Schwartz kernel $k(a,b;A)$ that is a distribution on $G\times G$. If $A$ is left $\Gamma$ invariant,  then 
\[
k(\ga a,\ga b;A)=k( a, b;A),
\]
and viceversa, and therefore $k( a, b;\A)$ may be viewed as a distribution on $(G\times G)/\Gamma$. We have the following result \cite{Ati}. Suppose that $A\in \B$ has a smooth kernel $k(a,b;A)$, and is positive and is self adjoint. Then, $A$ is of $\Gamma$-trace class, and
\[
\Tr_\Gamma A=\int_\FF k(g,g;A) d g,
\]

These two traces, $\Tr_G$ and $\Tr_\Gamma$, are indeed equivalent, up to a scalar factor. For  observe that the von Neumann algebra $\B=\RR_\Gamma\otimes B(L^2(\Gamma\bsl G))$ contains the von Neumann algebra $\RR_G$. 
Suppose that $A$ is an $G$-left invariant operator on $L^2(G)$. Whence, $A\in L(G)'=\RR_G\subseteq \B$, whence the $\Gamma$-trace of $A$ is defined. Note that $ R(L^1(G))\subseteq \RR_G$, so we have in $\RR_G$ the operators $R(f)$, with $f\in L^1(G)\cap L^2(G)$.  Since, 
\[
R(f)(u)(a)=\int_G f(g)R(g)(u)(a) dg=\int_G f(g)u(ag ) dg=\int_G f(a^{-1}t)u(t ) dt,
\]
$R(f)$ is the integral operator with kernel $k(a,g;\tilde T)=f(a^{-1}g)$. 
The $\Gamma$ trace of $R(f)$ is well defined and 
\[
\Tr_\Gamma R(f)=\int_{\Gamma\bsl G} k(e,e;\tilde T) dg
= k(e,e;\tilde T) \int_{\Gamma\bsl G} dg= \Vol (\Gamma\bsl G) f(e) =\Vol (\Gamma\bsl G) \Tr_G R(f).
\]

Since $R(L^1(G))$ is dense in $\RR_G$, it follows that for all operators $A\in \RR_G$:
\[
\Tr_\Gamma A= \Vol (\Gamma\bsl G) \Tr_G A.
\]

\subsubsection{The local trace}

We pass to give a suitable interpretation of the right side of equation (\ref{ee1}). The key point is a suitable description of the group Fourier transform and of the dual Haar measure. This may depend strongly on the particular group, however we pursue the general approach as far as possible.

We start by observing that the representation $\rho_h$ has an associated representation \cite{Gar} \cite[3.1]{Seg} \cite{Puk}
\[
d\rho_h:\gf\to \LL(\H),
\]
where $\LL(\H)$ is the Lie algebra of the skew symmetric operators on $\H$, and formally (whenever the limit exists)
\[
d\rho_h(V)(u)=\lim_{t\to 0}\frac{\rho_h(\e^{t V})(u)-u}{t}=\left.\frac{d}{dt}\rho_h(\e^{t V})(u)\right|_{t=0}.
\]

This representation extends to a representation of the universal enveloping algebra $\UU(\gf)$, and we will use this fact implicitly, without further comment or variation of the notation.   In particular, this means that we will apply $d\rho_h$ to polynomials on elements of $\gf$ (at most of degree two with real coefficients) just by linearity \cite[Theorem 3.1]{Gar}. 

We recall a few important properties of $d\rho_h$. Consider the subspace of $\H$ of the compactly supported $\CC^\infty$ vector fields:
\[
\H_0=\left\{v\in \H~|~ v=\rho_h(\vv)(u), ~\vv\in\CC^\infty_0(G), u\in \H\right\}.
\]

Then, if $V\in \gf$,  we have the following facts:
\begin{enumerate}
\item $\H_0$ is dense in $\H$,
\item $\H_0\leq \DS(d\rho_h(V))$, for all $V\in \gf$,
\item $d\rho_h(V)(\H_0)\leq \H_0$, for all $V\in \gf$,
\item $\rho_h(g)(\H_0)\leq \H_0$, for all $g\in G$,
\item the minimal closed extension of the restriction of $d\rho_h(V)$ to $\H_0$ is $d\rho(V)$.
\end{enumerate}

Now, take   a basis  $\{X_k\}$ of $\gf$, and let $\{x_k(t)=\ga_{V_k,e}(t)=\e^{t X_k}\}$ denote the corresponding local coordinate system on $G$ near $e$, i.e. the integral curves of the $X_k$ at $e$. The infinitesimal generators $d\rho_h(X_k)$ of $\rho_h$ at $g$ are defined by
\[
\e^{td\rho_h(X_k)}=\rho_h(\e^{t X_k}).
\]

Thus,
\beq\label{f1}
d\rho_h(X_k)=\left.\frac{d}{dt}\e^{td\rho_h(X_k)}\right|_{t=0}
=\left.\frac{d}{dt}\rho_h(\e^{t X_k})\right|_{t=0}=\b_{x_k}\rho_h(x)|_{x(0)}=X_k (\rho_h)(e).
\eeq


Consider the operator valued smooth function $\vv(g)\rho(g^{-1})$, $\vv\in \CC^\infty_0(G)$. Then
\begin{align*}
\b_{x_k}(\vv(x(g))\rho_h^\da(x(g)))&=
\b_{x_k}(\vv(x(g)))\rho_h^\da(x(g))+\vv(x(g))\b_{x_k}\rho_h^\da(x(g))\\
&=\b_{x_k}(\vv(x(g)))\rho_h^\da(x(g))+\vv(x(g))\b_{x_k}\rho^\da_h(x(g))\\
&=X_k (\vv)(g)\rho_h^\da(x(g))+\vv(g)X_k(\rho^\da_h)(g).
\end{align*}

Since
\[
d (\vv(x(g))\rho_h^\da(x(g))=\sum_k \b_{x_k}(\vv(x(g))\rho_h^\da(x(g)))dx_k,
\]
and by the Stockes theorem on $G$,
\[
\int_G d (\vv(g)\rho_h^\da(g)) =0,
\]
we find that
\beq\label{ee2}
\int_G X_k (\vv)(g)\rho_h^\da(g) dg=\int_G \vv(g)X_k(\rho^\da_h)(g)dg.
\eeq

We use these facts as follows. First, observe that the left side of equation (\ref{ee2}) is the group Fourier transform of $X_k(\vv)$:
\[
\FF (X_k(\vv))(h)=\int_G X_k (\vv)(g)\rho_h^\da(g) dg.
\]

Second, about the right side, note that as a function of $t$, 
\[
\frac{d}{dt}\rho_h(\e^{t X_k})=X_k (\rho_h)(g);
\]
but then, with $g=\e^{sX_k}$, 
\[
X_k (\rho_h)(g)=\left.\frac{d}{dt}\rho_h(\e^{(s+t) X_k})\right|_{t=0}=\left.\rho_h(g)\frac{d}{dt}\rho_h(\e^{t X_k})\right|_{t=0}
=\rho_h(g) X_k (\rho_h)(e).
\]

Whence
\[
\int_G \vv(g)X_k(\rho^\da_h)(g)dg=X_k(\rho^\da_h)(e)\int_G \vv(g)\rho(g^{-1}dg=-d\rho_h (X_k)\FF(\vv),
\]
and we have proved that
\beq\label{ee3}
\FF X_k \FF^{-1}=d\rho_h (X_k).
\eeq

\begin{prop} Let  $T$ be a positive operator in  $\RR$, then
\[
\Tr_\Gamma T=\Vol(\Gamma\bsl G)\int_{\hat G} \Tr d\rho_h(T)d\mu(h).
\]

Furthermore, if $T$ is a self adjoint integral operator with smooth kernel $k(h,g;T)$, then
\[
\int_\FF k(g,g;T) d g=\Vol(\Gamma\bsl G)\int^\oplus_{\hat G} \Tr d\rho_h(T)d\mu(h).
\]
\end{prop}

\section{Localised analytic torsion and relative analytic torsion}
\label{lt}

We are now ready to introduce the definitions of the main objects of interest in this work. 
Given an adjointable operator $S$ acting on  $L^2(G)$, and a representation $(\rho_h,\H_h)$,  $h\in \hat G$, we call the localisation of $S$ at $\rho_h$ the fibre $S(h)$ of the field of operators $h\mapsto \FF T\FF^{-1}(h)$ acting on $ \int^\oplus_{\hat G} \H_h\otimes \bar \H_h d\mu(h)$. We call localised spectrum of $S$ the spectrum of $S(h)$ (these definitions are  a particular instance of \cite[3.1]{BNPW}). Thus the localised spectrum of $S$ is the fibre of a field of spectra, in particular in the cases of interest, the fibre of a continuous field of sequences of real numbers with unique accumulation point at infinity.  Note that, by equation (\ref{ee3}), $S(h)=d\rho_h(S)$.

In order to proceed, we fix the left invariant Riemannian metric on $G$ that makes the basis $\{X_k\}$ orthonormal, and denote by $\star$ the induced Hodge operator. 
Since $G$ is a Lie group, we can fix  global bases, and we have the decomposition
\[
L^2(\gamma(G,\Lambda^\bu T^*G))=L^2(G)\otimes \Lambda^\bu T^*G.
\]

We denote by $\Omega^\bu(G)=\Gamma(G,\Lambda^\bu T^*G))$ the space of smooth sections, and by $\Omega^\bu_0(G)$ the space of the smooth sections with compact support. We denote by $d$ the minimal closed extension of the exterior derivative operator on $\Omega^\bu_0(G)$,  and by $(\DS(d^\bu), d^\bu)$ the associated de Rham complex of Hilbert spaces and closed operators. We denote  by $\de$ the adjoint of $d$, and by $\Delta=d\de+\de\d$ the associated Hodge Laplace operator. This is a self adjoint non negative operator with maximal domain the Sobolev space $H^2(G)$ (equivalently we may construct the formal Hodge Laplace operator  on the $\Omega^\bu_0(G)$, that is essentially self adjoint \cite{CY} 
\cite{MPR1}). 
In particular, observe that $\Delta$ is left invariant by construction. 

Since $G$ is the universal cover of the compact manifold $\Gamma\bsl G$, the heat operator $\e^{-t\Delta}$, $t>0$, is a bounded integral operator with smooth kernel by a form in $G\times G$ \cite{CY}. Whence, it is in $\RR_G$ and it has finite $G$ trace \cite{Ati}. In particular, by the Plancharel Theorem, this means that $d\rho_h(\e^{-t\Delta^{(q)}}) $ is of trace class, and hence $\Delta^q(h)$ has discrete spectrum, denoted by $\Sp \Delta^{(q)}(h)$. Therefore, 
\beq\label{ee4}
\Tr_\Gamma \e^{-t\Delta^{(q)}}=\int_{\hat G}^\oplus \Tr d\rho_h(\e^{-t\Delta^{(q)}}) d\mu(h),
\eeq
and
\beq\label{eee2}
\Tr d\rho_h(\e^{-t\Delta^{(q)}})=\sum_{\la(h)\in \Sp \Delta^{(q)}(h)} \e^{-t\la(h)}.
\eeq

The main analytic properties of the function $\Tr_\Gamma \e^{-t\Delta^{(q)}}$ as a function of $t$ were investigated in \cite{Lot}. In particular, it was proved that the behaviour for small $t$ is the same as that of the trace of the heat kernel of the Hodge Laplace operator on the compact quotient $\Gamma\bsl G$

In this setting, we introduce the following definitions. We call  zeta function of $\Delta^{(q)}$ localised at the representation $\rho_h$, $h\in \hat G$, the function of the complex variable $s$ defined for large $\Re(s)$ by the series
\[
\zeta(s,\Delta^{(q)}(h))=\sum_{\la(h)\in \Sp_+ \Delta^{(q)}(h)} \la(h)^{-s},
\]
and by analytic continuation elsewhere. Thus, $\zeta(s,\Delta^{(q)}(h))$ is the fibre of a field of functions on $\hat G$. The analytic properties of the localised zeta function are determined by the asymptotic for small and large $t$ of the trace of the localised heat operator by the Mellin transform
\[
\zeta(s,\Delta^{(q)}(h))=\frac{1}{\Gamma(s)}\int_0^\infty t^{s-1} \Tr \e^{-t\Delta_+^{(q)}(h)} dt.
\]

We call  analytic torsion zeta function of the group $G$ localised at the representation $\rho_h$, $h\in \hat G$, the graded sum
\beq
\label{loctorfun}
\tf(s;G,h)=\sum_{q=0}^m (-1)^q q\zeta(s,\Delta^{(q)}(h)),
\eeq
and assuming that this is regular at $s=0$, we call localised analytic torsion of $G$ the  complex vector field on $\hat G$
\beq
\label{loctor}
\TF(G;h)=\left.\frac{d}{ds}\tf(s;G,h)\right|_{s=0}.
\eeq

Next, let $\hat G_0$ denote the subspace of $\hat G$ determined by the following requirement: $h\in \hat G-\hat G_0$ if and only if there exists a real numbers $K_h$ and $K$ such that
\[
\Sp d\rho_h(\Delta^{(q)})>K_h>K>0,
\]
for all $q$. Then, we call relative analytic torsion of the pair $(G,\Gamma)$ the (possibly infinite) number
\beq
\label{reltor}
\TF_\Gamma(G)=\left. \frac{d}{ds}\int_{\hat G-\hat G_0}^\oplus \tf(s;G,h)d\mu(h)\right|_{s=0}+\int_{\hat G_0}^\oplus \TF(G;h)d\mu(h).
\eeq

It will be useful to introduce also the relative analytic torsion zeta function
\beq
\label{reltorfun}
\tf_\Gamma(s;G)=\int_{\hat G-\hat G_0}^\oplus \tf(s;G,h)d\mu(h),
\eeq
such that
\[
\TF_\Gamma(G)=\left. \frac{d}{ds}\tf_\Gamma(s;G) \right|_{s=0}+\int_{\hat G_0}^\oplus \TF(G;h)d\mu(h).
\]

This definition is clearly inspired by Definition 2 of \cite{Lot}, 
\[
\TT_\Gamma(G)=\frac{d}{ds}\left. \int_0^\ep t^{s-1}\frac{1}{\Gamma(s)} \Tr_\Gamma \e^{-t  \Delta_+} dt  \right|_{s=0}
+\int_\ep^\infty \frac{1}{t}\int_0^\infty\Tr_\Gamma \e^{-t  \Delta_+} dt.
\]
where 
\[
 \Tr_\Gamma \e^{-t  \Delta_+} =\sum_{q=0}^m(-1)^q q \Tr_\Gamma \e^{-t  \Delta_+^{(q)}},
 \]
and in fact we  prove that they are equivalent for a significant families of groups, see Proposition \ref{equiv}. This family is defined by the  assumptions introduced below. Before, some further notation, we will denote the graded sum of the trace of the heat operator by:
\[
\Tr \e^{-t d\rho_h \Delta}=\sum_{q=0}^m(-1)^q q \Tr \e^{-t d\rho_h \Delta^{(q)}}.
\]

\begin{ass} \label{ass}

\hfill \break

\begin{enumerate}

\item The space $\hat G$ has a continuous parameterisation by a  real variable $h$ in some interval $I$ of the  real line.

\item The Plancharel measure in this parameterisation reads $d\mu(h)= f(h) dh$, where $dh$ is the classical Lebesgue measure, and $f(h)$ some continuous (smooth) function integrable on $I$.

\item The eigenvalues $\la(h)$ of $d\rho_h \Delta$ are continuous (smooth) functions of $h$.

\item For large $th$ 
\[
\Tr \e^{-t d\rho_h \Delta}=O(\e^{th})=O(\e^{-t})O(\e^{-h}).
\]

\item For small $th$:
\[
\Tr \e^{-t d\rho_h \Delta}=\sum_{j,l=0}^{J,L} c_{j,l} t^{j-j_0} h^{l-l_0}+O(t^{J-j_0}) O(h^{L-l_0}),
\]
with real constants $c_{j,l}$, integers $j_0<J$, $l_0<L$,

\item for small $h$:
\[
f(h)=o(h^\ka),
\]
where $\ka=max (l_0-1,0)$.

\end{enumerate}

\end{ass}

Note that the requirement on the large $t$ behaviour of $\Tr_\Gamma \e^{-t\Delta}$ assumed in \cite{Lot}, Note 1.3, follows by these assumptions.

Note also that the measurable fields introduced above are continuous under assumptions (1) and (2).

\begin{prop}\label{asym} The localised torsion zeta function $\tf(s;G,h)$ of $G$ is analytic for $s>j_0$,  and has an analytic continuation to $\C$ with possible simple poles located at the integers points of type $j-j_0$, $j=0,1,2,\dots$, $j\not= j_0$. 
 \end{prop}
 \begin{proof} This follows by assumption (4), using classical methods (see for example \cite{Gil} or \cite{Spr2}).
 \end{proof}

\begin{prop} Let $G_1$ and $G_2$ be two Lie groups as above, and $[\rho_{1,h_1}]\in \hat G_1$, $[\rho_{2,h_2}]\in \hat G_2$. Then, (with a little obvious change of notation)
\[
\TF(G_1\times G_2, \rho_{1,h_1}\otimes \rho_{2,h_2})=\chi(G_1,\rho_{1,h_1}) \TF(G_2,\rho_{2,h_2})+\chi(G_2,\rho_{2,h_2}) \TF(G_1,\rho_{1,h_1}),
\]
where $\chi(G,\rho_h)=\sum_{q=0}^{\dim G} (-1)^q \dim \ker d\rho_h \Delta^{(q)}$.
\end{prop}
\begin{proof} This follows as in the proof of Theorem 2.3 of \cite{RS} or point (2) pg. 268 of \cite{Mul}. \end{proof}

\begin{prop}\label{equiv} When both defined, $\TF_\Gamma(G)=\TT_\Gamma(G)$.
\end{prop}
 
\begin{proof} By Assumptions (1), (2) and (3), and identifying $I$ with $(0,\infty)$, we may write
\[
\TF_\Gamma(G)=\left. \frac{d}{ds}\int_\de^\infty \tf(s;G,h) f(h) dh\right|_{s=0}+\int_0^\de \TF(G,h)f(h) dh.
\]

By Assumption (4), 
\[
\int_\de^\infty \int_\ep^\infty t^{s-1} \Tr \e^{-t d\rho_h \Delta_+} dt f(h) dh
=\int_\de^\infty f(h) O(\e^{-h}) dh\int_\ep^\infty t^{s-1}O(\e^{-t}) dt,
\]
is a regular analytic function of $s$ near $s=0$, and therefore (recalling that $\frac{1}{\Gamma(s)}=s+O(s^2)$),

\beq
\label{pop}
\frac{d}{ds}\left.\int_\de^\infty \frac{1}{\Gamma(s)}\int_\ep^\infty t^{s-1} \Tr \e^{-t d\rho_h \Delta_+} dt f(h) dh\right|_{s=0}
=\int_\ep^\infty \frac{1}{t}\int_\de^\infty\Tr \e^{-t d\rho_h \Delta_+} dt f(h) dh.
\eeq

By Proposition \ref{asym},  for fixed $k$ and $\Re(s)>j_0$, the function 
\beq\label{ep1}
\tf(s;G,h)=\frac{1}{\Gamma(s)}\int_0^\infty t^{s-1} \Tr \e^{-t d\rho_h \Delta_+} dt,
\eeq
has an analytic extension regular at $s=0$:
\[
\tf(s;G,h)=\tau(0;G,h)+\tf'(0;G,h) s+o(s),
\]
near $s=0$, where
\[
\TF(G,h)=\tf'(0;G,h)=\left. \frac{d}{ds}\tf(s;G,h)\right|_{s=0}.
\]

Consider now
\[
\TF_\Gamma(G)=\left. \frac{d}{ds}\int_\de^\infty \tf(s;G,h) f(h) dh\right|_{s=0}+\int_0^\de \TF(G,h)f(h) dh.
\]

For the first term, using equation (\ref{ep1}), for $\Re(s)>j_0$
\begin{align*}
\int_\de^\infty \tf(s;G,h) k dk
=&\int_\de^\infty \frac{1}{\Gamma(s)}\int_0^\ep t^{s-1} \Tr \e^{-t d\rho_h \Delta_+} dt f(h) dh\\
&+\int_\de^\infty \frac{1}{\Gamma(s)}\int_\ep^\infty t^{s-1} \Tr \e^{-t d\rho_h \Delta_+} dt f(h) dh,
\end{align*}
using equation (\ref{pop}), and adding and subtracting the suitable term,
\begin{align*}
\frac{d}{ds}\left.\int_\de^\infty \tf(s;G,h) f(h) dh \right|_{s=0}
=&\frac{d}{ds}\left. \int_0^\ep t^{s-1}\frac{1}{\Gamma(s)} \Tr_\Gamma \e^{-t  \Delta} dt  \right|_{s=0}\\
&+\int_\ep^\infty \frac{1}{t}\int_\de^\infty\Tr \e^{-t d\rho_h \Delta} dt f(h) dh\\
&- \frac{d}{ds}\left.\int_0^\de \frac{1}{\Gamma(s)}\int_0^\ep t^{s-1} \Tr \e^{-t d\rho_h \Delta} dt f(h) dh\right|_{s=0}.
\end{align*}

For the second term, 
\begin{align*}
\int_0^\de \TF(G,h)k dk=&\int_0^\de \frac{d}{ds}\left. \tau(s;G,h) \right|_{s=0} f(h) dh\\
=&\int_0^\de \frac{d}{ds}\left. \frac{1}{\Gamma(s)}\int_0^\infty t^{s-1} \Tr \e^{-t d\rho_h \Delta} dt  \right|_{s=0} f(h) dh\\
=& \frac{d}{ds}\left. \frac{1}{\Gamma(s)}\int_0^\de\int_0^\infty t^{s-1} \Tr \e^{-t d\rho_h \Delta}   dt f(h) dh \right|_{s=0}.
\end{align*}
because $f$ is continuous, $h$-integration commutes with both $s$-derivation and $t$-integration. 
Now, split the $h$-integral at $h=\de$, and consider the two terms separately. Keep the first one and rewrite the second one as follows. First note that, 
by assumption (4), for fixed $h$ and large $t$,
\begin{align*}
\int_\ep^\infty t^{s-1} \Tr \e^{-t d\rho_h \Delta} dt =&O\left(\int_\ep^\infty t^{s-1}  \e^{-ht} dt\right)=O\left(h^{-s}\int_{h\ep}^\infty x^{s-1} \e^{-x} dt\right)\\
=&O\left(h^{-s}\int_0^\infty x^{s-1} \e^{-x} dt\right)=\Gamma(s) O(h^{-s}).
\end{align*}

Then, by Assumption (6),
\[
 \frac{1}{\Gamma(s)}f(h)\int_\ep^\infty t^{s-1} \Tr \e^{-t d\rho_h \Delta} dt = O(k^{\ka-s}),
 \]
 and, as a function of $s$, has an expansion near $s=0$ with coefficients that are continuous integrable functions of $h$ in   $(0,\de)$. It follows that
\[
\int_0^\de \frac{d}{ds}\left. \frac{1}{\Gamma(s)}\int_\ep^\infty t^{s-1} \Tr \e^{-t d\rho_h \Delta} dt  \right|_{s=0} f(h) dh
=\int_\ep^\infty \frac{1}{t} \int_0^\de \Tr \e^{-t d\rho_h \Delta} f(h) dh dt,
\]
and then
\begin{align*}
\int_0^\de \TF(G,h)f(h) dh=&\frac{d}{ds}\left. \frac{1}{\Gamma(s)}\int_0^\de\int_\ep^\infty t^{s-1} \Tr \e^{-t d\rho_h \Delta}   dt f(h) dh \right|_{s=0}\\
&+\int_\ep^\infty \frac{1}{t} \int_0^\de \Tr \e^{-t d\rho_h \Delta} f(h) dh dt.
\end{align*}

Collecting
\begin{align*}
\TF_\Gamma(G)=&\left. \frac{d}{ds}\int_\de^\infty \tf(s;G,h) k dk\right|_{s=0}+\int_0^\de \TF(G,h)f(h) dh\\
=&\frac{d}{ds}\left. \int_0^\ep t^{s-1}\frac{1}{\Gamma(s)} \Tr_\Gamma \e^{-t  \Delta} dt  \right|_{s=0}
+\int_\ep^\infty \frac{1}{t}\int_\de^\infty\Tr \e^{-t d\rho_h \Delta} dt f(h) dh\\
&- \frac{d}{ds}\left.\int_0^\de \frac{1}{\Gamma(s)}\int_0^\ep t^{s-1} \Tr \e^{-t d\rho_h \Delta} dt f(h) dh\right|_{s=0}\\
&+\frac{d}{ds}\left. \frac{1}{\Gamma(s)}\int_0^\de\int_\ep^\infty t^{s-1} \Tr \e^{-t d\rho_h \Delta}   dt f(h) dh \right|_{s=0}
+\int_\ep^\infty \frac{1}{t} \int_0^\de \Tr \e^{-t d\rho_h \Delta} f(h) dh dt\\
=&\frac{d}{ds}\frac{1}{\Gamma(s)}\left. \int_0^\ep t^{s-1} \Tr_\Gamma \e^{-t  \Delta} dt  \right|_{s=0}
+\int_\ep^\infty \frac{1}{t}\int_0^\infty\Tr \e^{-t d\rho_h \Delta} dt f(h) dh.
\end{align*}

\end{proof}
 
\section{The abelian case}
\label{ab}

In this section we apply our construction to the simplest case of the real number field $G=\R$ considered as a Lie group with respect to the addition. This is a connected simply connected locally compact second countable space, and an abelian group (actually it is contractible and separable). The discrete subgroup $\Gamma=\Z$ has compact quotient $\T=\R/\Z$. The irreducible representations are one dimensional, i.e. the characters 
\begin{align*}
\chi_h&:\R\to U(\C)=U(1),&
\chi_h&:g\to \e^{2\pi i h g },
\end{align*}
with $h\in \R$. The action $\chi_h$ restricts to an action of $\Z$ with fundamental domain $[0,1]$. The dual $\hat \R$ of $\R$ is  isomorphic to $\R$ by the paring $\hat g_h (g')=\e^{2\pi i h g'}$. The Plancharel measure is the Lebesgue measure $dh$, the direct integrals are Lebesgue integrals on $\R$, and the group Fourier transform is the classical Fourier transform
\begin{align*}
\hat f(h)&=\int_{\R} f(g)\e^{-2\pi i h g} dg,&
f(g)&=\int_{ \R} \hat f(h) \e^{2\pi i h g} dh.
\end{align*}

As a differentiable manifold, we have a global coordinate $\{x\}$, with coordinate basis of the tangent space $\{\b_x\}$, with dual $\{dx\}$. These are left invariant vector fields. The formal exterior differential on $\CC^\infty(\R)$  is $d u=\frac{d}{dx} u dx$, 
and as usual we use the same notation for the minimal closed extension of its restriction on $\CC_0^\infty(\R)$. A left invariant Riemannian metric is $dx\otimes dx$, and with respect to it the previous bases are orthonormal. The Hodge star $\star dx=1$, the volume form $dx$. The inner product 
$
\lp \om,\vv\rp=\int_\R \om\wedge \star\vv dx
$, 
and the Hodge Laplace operator (we just need it on functions)
\begin{align*}
\Delta^{(0)} \om&=-\frac{d^2}{dx^2}\om.
\end{align*}

The heat operator $\e^{-t\Delta^{(0)}}$ is the integral  operator with smooth kernel \cite[pg. 6]{Ros}
\[
k(x,y;\e^{-t\Delta^{(0)}})=\frac{1}{\sqrt{4\pi t}}\e^{-\frac{|x-y|^2}{4t}}.
\]

The $\Gamma=\Z$ trace is
\[
\Tr_\Z \e^{-t\Delta^{(0)}}=\int_0^1 k(x,x;\e^{-t\Delta^{(0)}})dx=\int_0^1 \frac{1}{\sqrt{4\pi t}}dx=\frac{1}{\sqrt{4\pi t}}.
\]

In order to compute the  $\CC^\infty$ vectors, using equation  (\ref{f1}), we need to isolate the case $h=0$. For $h\not=0$: 
$d\chi_h (\b_x)=2\pi i h$. Thus, 
\[
d\chi_h(\Delta^{(0)})=4\pi^2 h^2,
\]
and
\[
\Sp d\chi_h(\Delta^{(0)})=\Sp d\chi_h(\Delta^{(1)})=\{4\pi^2 h^2\}.
\]

When $h=0$, $\chi_0=1$, the constant representation, so the tangent vector is the zero vector, and the Laplace operator is multiplication by $0$, that has one eigenvalue $0$ with infinite multiplicity. However, the space $\{0\}$ has measure zero, whence we ignore this representation, and we proceed assuming $h\not=0$.

The heat operator localised at $\chi_h$, $h\not=0$  is
\[
d\chi_h(\e^{-t\Delta^{(0)}})=\e^{-4\pi^2 h^2t},
\]
and its trace is
\[
\Tr d\chi_h(\e^{-t\Delta^{(0)}})=\e^{-4\pi^2 h^2t}.
\]

The $\Gamma=\Z$ equivariant trace is 
\[
\Tr_\Z \e^{-t\Delta^{(0)}}=\Vol(\Gamma\bsl G)\int_{\R-\{0\}} \e^{-4\pi^2 h^2t} dh=2\int_0^\infty \e^{-4\pi^2 h^2t} dh=\frac{1}{2\sqrt{\pi t}},
\]
since $\Vol(\Gamma\bsl G)=1$. The localised analytic torsion zeta function is ($h\not=0$)
\[
\tf(s;\R,h)=-(2\pi h)^{-2s},
\]
that is holomorphic for all $s$. 
Whence

\[
\TF(\R, h)=2\log 2\pi |h|.
\]

According to equation (\ref{reltorfun}), the relative torsion zeta function, for large $s$, is
\begin{align*}
\tf_\Z(s;\R)&=-2\int_\ep^\infty (2\pi h)^{-2s} dh 
=2\frac{(2\pi \ep)^{1-2s}}{2\pi (1-2s)},
\end{align*}
that gives
\begin{align*}
T_\Z(\R)&=\tf_\Z'(0;\R)+2\int_0^\ep 2\log (2\pi h) dh=4\ep-4\ep \log 2\pi \ep  +4\ep \log 2\pi \ep -4\ep=0.
\end{align*}

On the other side (where the first term is defined for large $s$)
\[
\TT_\Gamma=\left.\frac{d}{ds}\frac{1}{\Gamma(s)}\int_0^\ep t^{s-1}\frac{1}{2\sqrt{\pi t}} dt \right|_{s=0}+\int_\ep^\infty \frac{1}{t} \frac{1}{2\sqrt{\pi t}} dt
=0.
\]

Next, we compare this with the classical analytic torsion. The irreducible representations $\chi_\alpha:\pi_1(\T)=\Z\to U(1)$ of the fundamental group of the circle are the characters 
$\e^{2\pi i\alpha n}$, with $0\leq \alpha <1$. Fix $\al\not=0$. The Hodge Laplace operator on function with values in $E_\al$ is positive definite, and its spectrum  is $\Sp \Delta^{(0)}=\{(n+\al)^2\}_{n\in \Z}$. 
The torsion zeta function of $\T$ with coefficients twisted by $\chi_\al$,  is 
\beq\label{ctc}
\begin{aligned}
t(s;\T,\al)&=\sum_{q=0}^1(-1)^q q\zeta(s,\Delta_+^{(q)})=-\zeta(s,\Delta_+^{(1)})=
-(2\pi)^{-2s}\sum_{n\in \Z} (n+\al)^{-2s}\\
&=-(2\pi)^{-2s}\zeta_H(2s,\al)-(2\pi)^{-2s}\zeta_H(2s,1-\al).
\end{aligned}
\eeq

Recalling that $\zeta_H(0,q)=\frac{1}{2}-q$, and $\zeta_H'(0,q)=\log\Gamma(q)-\frac{1}{2}\log 2\pi$, we compute 
the analytic torsion of $\T$ with coefficients twisted by $\chi_\al$:
\[
T(\T,\al)=t'(0,\al)=2\log 2\sin \pi\al.
\]

Reconsider the definition of the relative zeta function of $\R$, with $\de=1$:
\begin{align*}
\tf_\Z(s;\R)&=-2\int_1^\infty (2\pi h)^{-2s} dh =-\int_{-\infty}^{-1}(2\pi h)^{-2s} dh-\int_{1}^{\infty}(2\pi h)^{-2s} dh\\
&=-\sum_{n=-\infty}^{-2}\int_{0}^{1}(2\pi (n+\al))^{-2s} d\al-\sum_{n=1}^{\infty}\int_{0}^{1}(2\pi (n+\al))^{-2s} d\al\\
&=-\int_{0}^{1}\sum_{n=2}^{\infty}(2\pi (n-\al))^{-2s} d\al-\int_{0}^{1}\sum_{n=1}^{\infty}(2\pi (n+\al))^{-2s} d\al\\
&=\int_{0}^{1} t(s;\T,\al) d\al+2\int_0^1 (2\pi \al)^{-2s} d\al,\\
&=\int_{0}^{1} t(s;\T,\al) d\al-2\int_0^1 \tf(s;\R,\al) d\al.
\end{align*}
where the last equivalence follows because the series of function converges uniformly in $\al$ for $\al\in [0,1)$ when $\Re (s)$ is large. 
Whence, collecting and using equation (\ref{ctc}),
\begin{align*}
\tf_\Z(s;\R)&=\int_{0}^{1} t(s;\T,\al) d\al+2\int_{0}^{1}(2\pi \al)^{-2s} d\al=\int_0^1 
t(s;\T,\al) d\al+\frac{2(2\pi)^{-2s}}{1-2s}.
\end{align*}

Thus, 
\begin{align*}
\tf_\Z(s;\R) 
=&-\int_{0}^{1}(2\pi)^{-2s} \zeta_H(2s,1-\al)d\al-\int_{0}^{1}(2\pi)^{-2s} \zeta_H(2s,\al)d\al+2\int_{0}^{1}(2\pi \al)^{-2s} d\al.
\end{align*}

Using for example the Hermite representation for the Hurwitz zeta function, we find that
\begin{align*}
-\int_{0}^{1}(2\pi)^{-2s} \zeta_H(2s,1-\al)d\al 
=-(2\pi)^{-2s} \int_0^1 \left(\frac{(1-\al)^{1-2s}}{2s-1}+f(s,\al)\right) d\al, 
\end{align*}
where $f(s,\al)$ is a regular analytic function of $s$ for all $s$, smooth and bounded in $\al$ for $\al\in [0,1]$. Therefore, for $s$ near $s=0$,
\begin{align*}
-\frac{d}{ds}\int_{0}^{1}(2\pi)^{-2s} \zeta_H(2s,1-\al)d\al 
=-\int_{0}^{1} \frac{d}{ds}(2\pi)^{-2s} \zeta_H(2s,1-\al)d\al.  
\end{align*}

According to the definition, equation (\ref{reltor}), this gives

\begin{align*}
\TF_\Z(\T)&=
\left.
\frac{d}{ds} \tf_\Z(s;\T)\right|_{s=0}+2\int_0^1 \TF(\T,\al)d\al\\
&=\int_0^1 \left.
\frac{d}{ds}t(s;\T,\al)\right|_{s=0} d\al-4\log 2\pi +4+2\int_0^1 \TF(\T,\al)d\al\\
&=\int_0^1 T(\T,\al)d\al.
\end{align*}

Recalling the multiplicative property of all the torsions, we have proved the following result (where we are assuming that $G$ has trivial compact factor).

\begin{prop} Let $G$ be a simply connected abelian real Lie group, with discrete co compact subgroup $\Gamma$,  then
\[
\TF_\Gamma(G)=\int_0^1 T(G,\al)d\al.
\]
\end{prop}

\section{The Heisenberg group}
\label{H}

In this section is devoted to our main application, i.e. the three dimensional Heisenberg group.

\subsection{The Heisenberg group}

There exist several equivalent definitions of the Heisenberg group, we chose the more suitable for our purpose. We call Heisenberg group $H$ the three dimensional real space with the Lie group  operation (this group is called polarised Heisenberg group by Folland \cite[pg. 19]{Fol}, see also \cite[pg. 47]{CG})
\[
(a,b,t)(a',b',t')=(a+a',b+b', t+t'+ab').
\]

The discrete subgroup $\Gamma$ is 
\[
\Gamma=\{(l,m,n)\in H~|~l,m,n\in \Z\}.
\]

Topologically,  the Heisenberg group is a contractible space, locally compact and second countable, and as a Lie group is nilpotent unimodular of type I \cite[Ex. 3, pg. 229]{Fol1}. As a real smooth manifold, $H$ has a global coordinate system $\{a,b,t\}$, with coordinate basis of $T_e H$: $\{\b_a,\b_b,\b_t\}$. The corresponding basis of  left invariant vector fields in the Lie algebra $\hf$ is 
\begin{align*}
X(g)&=\b_a,&
Y(g)&=\b_b+a\b_t,&
T(g)&=\b_t,
\end{align*} 
with $[X,Y]=T$, and dual basis $\{da,db,\te\}$, where $\te= dt-adb$. 
The formal exterior derivative operator on functions $\om$ and one forms $\om=\om_a da +\om_b db+\om_\te \te$ is
\begin{align*}
d \om&=\b_a \om da+\b_b\om db+\b_t\om dt=X\om da+Y\om db+T\om \te,\\
d\omega &= \left( X\omega_b - Y\omega_a - \omega_\theta \right)da \wedge db 
 +\left(X\omega_\theta  -T\omega_a \right) da \wedge \theta 
 +\left(Y\omega_\theta  -T\omega_b \right) db \wedge \theta .
\end{align*}

We fix the Riemannian  structure on $H$ determined by the left invariant Riemannian metric making the bases above orthonormal, that reads
\[
 da\otimes da+\left(1+a^2\right) db\otimes db +dt\otimes dt-a da\otimes db -a db\otimes da,
\]
with volume form $d{\it vol}=da\wedge db\wedge \te$. The  Hodge star $\star$  follows easily.

The (formal) Hodge Laplace operator on $q$-forms, $\Delta^{(q)}=\de d+d \de$, where $\de=-\star d \star$,  in terms of the orthogonal basis has the following explicit description (in the relevant degrees) \cite{Sch} \cite{MPR} \cite{MPR1}:
\begin{align*}
\Delta^{(0)} \om =&-(X^2 +Y^2 +T^2 )\om,\\
\Delta^{(1)}(\om_a da+\om_b db+\om t \te)=&(\Delta^{(0)} \om_a -T\om_b-Y\om_t) da+(\Delta^{(0)} \om_b +T\om_a-X\om_t) db\\
&+(\Delta^{(0)} \om_t +Y\om_a-X\om_b+\om_t) \te.
\end{align*}

\subsection{Dual group, irreducible representations and localised operators}

The dual space $\hat G$ (with the Fell topology \cite{Fell}) is homeomorphic with the set of the co adjoint orbits of $\hf^*$, with the natural quotient topology, where the last may be identified with  real line private  by the origin  \cite[Theorem 7.9, and example 7.6.1]{Fol1}. Thus, assumption (1) is satisfied with $I=\R-\{0\}$. The Plancharel measure is $d\mu(h)=|h|dh$, $h\in I$, and thus also assumption (2) is satisfied. 
The group Fourier transform on $H$ may be described explicitly as \cite[pg. 43]{Fol}:
\[
\Tr \hat f(h) \rho_h(a,b,t) =\frac{1}{|h|} \int_{\R} f(a,b,t)\e^{-2\pi i h(t-s)} ds.
\]

The irreducible (infinite dimensional, the finite dimensional ones do not matter since they determine a set of measure zero in $\hat G$ ) unitary representations of the Heisenberg group are the  Schr\"{o}dinger representations $\rho_h:H\to U(\SS)\subseteq U(L^2(\R))$, 
\[
\rho_h(a,b,t)(f)(x)=\e^{2\pi i ht+2\pi ibx} f(x+h a),
\]
on $L^2(\R)$, where  $h\in \R-\{0\}$ is the parameter described above for the dual group $\hat G$ \cite[3(1.25), pg. 22, Th. 1.59, pg. 37]{Fol} (see also equation (4.1), pg. 47 of \cite{CG}). 
The infinitesimal generators of the associated representations of universal enveloping algebra $\UU(\hf)$ are
\begin{align*}
d\rho_h(X)&=h\b_x,&
d\rho_h(Y|_e)&=2\pi i x,&
d\rho_h(T)&=2\pi i h,
\end{align*}
that gives the (continuous ) field of operators $d\rho_h \Delta^{(q)}$. Our aim is to give an explicit description of the spectrum of these operators. For functions, this is quite easy, see next section. A similar calculation for $\Delta^{(1)}$ is more involved, so we prefer to follow the easier alternative approach delineated in \cite{MPR} and based on the Bergmann Fox representation, described in   Section \ref{Barg}.

\subsection{The localised Hodge Laplace operator on functions and its spectrum}

A direct caculation gives
\[
d\rho_h(\Delta^{(0)})=-h^2 \b_x^2+4\pi^2 x^2+4\pi^2 h^2.
\]

We aim to determine a spectral resolution of $\Delta^0$. For consider the eigenvalues equation

\beq\label{e1}
- \b_x^2f+\frac{4\pi^2}{h^2} x^2f=\left(\frac{\la}{h^2}-4\pi^2\right) f,
\eeq

Let $u_n(x)=H_n(x)$ the Hermite polynomial. It satisfies the differential equation 
\[
u_n''-2x u_n'+2n u_n=0.
\]

By the Liouville transform $u_n(x)=\e^\frac{x^2}{2}v_n(x)=k(x)v_n(x)$, with
\begin{align*}
u&=kv,& u'&=k'v+kv',& u''&=k''+2k'v'+kv,
\end{align*}
we end up with 

\[
-v_n''+x^2 v_n=(2n+1)v_n.
\]

Set $x=\sqrt{a} t$,  $f_n(t)=v_n(\sqrt{a}t)$, and $a=\frac{2\pi}{|h|}$, then $f$ satisfies
\[
- f''_n-\frac{4\pi^2}{h^2} t^2 f_n=(2n+1)\frac{2\pi}{|h|} f_n.
\]

Comparing the last equation with equation (\ref{e1}), we find:
\begin{align*}
f(x)&=v_n(\sqrt{2\pi/h}x)=\e^{-\frac{\pi x^2}{|h|}}H_n(\sqrt{2\pi/|h|}x)=(-1)^n\e^\frac{\pi x^2}{|h|} \frac{d^n}{dx^n} \e^{-\frac{2\pi x^2}{|h|}},\\
 \la&=\la_n=2\pi |h|(2n+1)+4\pi^2 h^2,
\end{align*}
with $n=0,1,2,\dots$, and $h\in \R-\{0\}$ (compare with \cite[Section 7, pg. 51]{Fol}).

\subsection{The Bargmann Fox representation and the spectrum of the localised Hodge Laplace operator on one forms}
\label{Barg}

Fix $h>0$, and let
\[
\|f\|_{\FF_h}^2=h\int_\C |f(z)|^2 \e^{-\pi h|z|^2} dz,
\]
and 
\[
\FF_h=\{f:\C\to \C, ~f~ {\it entire}, ~\|f\|_{\FF_h}<\infty\},
\]
then, the Bargmann transform is the map \cite[pg. 47]{Fol}
\begin{align*}
B_h&:  L^2(\R, dx)\to B_h(L^2(\R))=\FF_h\subseteq L^2(\C, h \e^{-\pi h |z|^2} dz),&
B_h&:f\mapsto B_h(f),
\end{align*}
where
\[
B_h(f)(z)=\left(\frac{2}{h}\right)^\frac{1}{4} \int_{-\infty}^{+\infty} f(x) \e^{2 \pi x z-\frac{\pi}{h} x^2-\frac{\pi h}{2} z^2} dx,
\]
and is a unitary map intertwining the Shrodinger representation with the Bargmann representation 
$\be_h B_h=B_h\rho_h$. To see this is better to complexify all the construction. Setting $z=a+ib$, we find

\begin{align*}
\be_h(z,t)(f)(w)&=\e^{2\pi i h t -\frac{\pi h}{2}|z|^2-\frac{\pi h}{4}(z^2-\bar z^2) -\pi h w\bar z} f(z+w).
\end{align*}

If $h<0$, the representation space is
\[
\bar\FF_h=\{f~|~fj \in \FF_{-h}\},
\]
where $j$ denotes conjugation, i.e. $j(z)=\bar z$;  the Bargmann transform is the map 
\begin{align*}
\bar B_h&:  L^2(\R, dx)\to \bar \FF_h\subseteq L^2(\C, |h| \e^{\pi h |z|^2} dz),&
\bar B_h&:f\mapsto B(f),
\end{align*}
where
\[
\bar B_h(f)(z)=\left(\frac{2}{|h|}\right)^\frac{1}{4} \int_{-\infty}^{+\infty} f(x) \e^{-2 \pi x z+\frac{\pi}{h} x^2+\frac{\pi h}{2} z^2} dx.
\]
and a direct calculation as the previous one gives 
\begin{align*}
\bar \be_{h}(z,t)(f)(w)&=\e^{2\pi i h t +\frac{\pi h}{2}|z|^2-\frac{\pi h}{4}(z^2-\bar z^2) +\pi h w z} f(\bar z+w).
\end{align*}

We compute the infinitesimal generators of the coordinate basis vectors:
\begin{align*}
h&>0:&d\be_h(\b_z)&=\b_w,&
d\be_h(\b_{\bar z})&=-\pi h w,&
d\be_h(\b_t)&=2\pi i h,\\
h&<0:&
d\bar\be_h(\b_z)&=h\pi w,&
d\bar\be_h(\b_{\bar z})&=\b_ w,&
d\bar\be_h(\b_t)&=2\pi i h.
\end{align*}

An orthonormal basis of left invariant vector fields is
\begin{align*}
Z&=\sqrt{2}\b_z-\frac{i}{2\sqrt{2}}(z+\bz) \b_t,&
\BZ&=\sqrt{2}\b_\bz+\frac{i}{2\sqrt{2}}(z+\bz) \b_t,&T,
\end{align*}
with $[Z,\BZ]=i [X,Y]=i T$, and
\begin{align*}
X&= \frac{1}{\sqrt{2}}(Z+\BZ),&
Y&=\frac{i}{\sqrt{2}}(Z-\BZ);
\end{align*}
the dual basis is
\begin{align*}
\ze&=\frac{1}{\sqrt{2}}dz=\frac{1}{\sqrt{2}}(da+idb),&\bze&=\frac{1}{\sqrt{2}}d\bz=\frac{1}{\sqrt{2}}(da-idb),&&\be=dt+\frac{i}{2\sqrt{2}}(z+\bz)( \ze-\bze)=\te,
\end{align*}

Then, we compute
\[
\Delta^{(0)}=-X^2-Y^2-T^2=-Z\BZ-\BZ Z-T^2.
\]
and
\begin{align*}
\Delta^{(1)}=&(-(Z\BZ+\BZ Z+T^2)\om_z-iT\om_z-iZ \om_t) \ze\\
&+(-(Z\BZ+\BZ Z+T^2)\om_\bz+iT\om_\bz+i\BZ \om_t) \bze\\
&+(-(Z\BZ+\BZ Z+T^2+1)\om_t-i\BZ\om_z+iZ \om_\bz) \te.
\end{align*}

The infinitesimal generators read 
\begin{align*}
h&>0:&d\be_h(Z)&=\sqrt{2}\b_w,&
d\be_h(\BZ)&=-\sqrt{2}\pi h w,&
d\be_h(\b_t)&=2\pi i h,\\
h&<0:&d\bar\be_h(Z)&=\sqrt{2}h\pi w,&
d\bar\be_h(\BZ)&=\sqrt{2}\b_w,&
d\be_h(\b_t)&=2\pi i h,
\end{align*}
and, setting  $k=2\pi h$, the localised Hodge Laplace operator is
\begin{align*}
d\be_h(\Delta^{(1)})
=&\left(\begin{array}{ccc} 2kw\b_w+k^2 +2k&0&-\sqrt{2}i\b_w\\ 
0& 2k w\b_w+k^2&-\frac{ik}{\sqrt{2}} w\\
\frac{ik}{\sqrt{2}}w&\sqrt{2}i\b_w& 2kw\b_w+k+k^2+1
\end{array}
\right)\end{align*}
and
\begin{align*}
d\bar\be_h(\Delta^{(1)})
=&\left(\begin{array}{ccc} -2kw\b_w+k^2 &0&-\frac{ik}{\sqrt{2}} w\\ 
0& -2k w\b_w+k^2-2k&\sqrt{2}i\b_w\\
-\sqrt{2}i\b_w&\frac{ik}{\sqrt{2}}w& -2kw\b_w-k+k^2+1
\end{array}
\right).
\end{align*}

We want to compute the spectrum of these operators. Since the analysis for positive and negative $h$ is analogous, we give details for $h>0$. First, observe that
\begin{align*}
d\be_h(\Delta^{(1)})-xI
=&\left(\begin{array}{ccc} 2kw\b_w+k^2+k &0&0\\ 
0& 2k w\b_w+k^2+k&0\\
0&0& 2kw\b_w+k^2+k
\end{array}
\right)\\
&+\left(\begin{array}{ccc}k-x&0&-\sqrt{2}i\b_w\\ 
0& -k-x&-\frac{ik}{\sqrt{2}} w\\
\frac{ik}{\sqrt{2}}w&\sqrt{2}i\b_w&1-x
\end{array}
\right),
\end{align*}

Next, observe that the homogeneous polynomial belong to $\FF_h$, more precisely  the set of the monomials
\[
\chi_l(w)=\frac{(h\pi)^\frac{l}{2}}{\sqrt{l!}}w^l,
\]
$l=0,1,2,\dots$, is an orthonormal basis of $\FF_h$ \cite[(1.63), pg. 40]{Fol}.  Thus, apply the Laplacian to the  vectors
$
(A\chi_{l_1}(w),B\chi_{l_2}(w),C\chi_{l_3}(w))
$. We end up with the system of equations
\[
\left\{
\begin{array}{l}
(2kl_1+k^2+2k)\frac{(h\pi)^\frac{l_1}{2}}{\sqrt{l_1!}}w^{l_1} A-\sqrt{2} i \frac{(h\pi)^\frac{l_3}{2}}{\sqrt{l_3!}} l_3 C w^{l_3-1}
=x \frac{(h\pi)^\frac{l_1}{2}}{\sqrt{l_1!}} w^{l_1} A,\\
(2kl_2+k^2)\frac{(h\pi)^\frac{l_2}{2}}{\sqrt{l_2!}}w^{l_2} B-\frac{ik}{\sqrt{2}} \frac{(h\pi)^\frac{l_3}{2}}{\sqrt{l_3!}} C w^{l_3+1}
=x \frac{(h\pi)^\frac{l_2}{2}}{\sqrt{l_2!}} w^{l_2} B,\\
\frac{ik}{\sqrt{2}} \frac{(h\pi)^\frac{l_1}{2}}{\sqrt{l_1!}} A w^{l_1+1}+\sqrt{2} i \frac{(h\pi)^\frac{l_2}{2}}{\sqrt{l_2!}} l_2 w^{l_2-1}B
+(2kl_3+k^2+k+1)\frac{(h\pi)^\frac{l_3}{2}}{\sqrt{l_3!}}w^{l_3}C
=x \frac{(h\pi)^\frac{l_3}{2}}{\sqrt{l_3!}} w^{l_3} C.
\end{array}
\right.
\]

We have the particular solutions: if $A=C=0$, then $l_2=0$ and we have the equation
\[
k^2 B=xB,
\]
that gives $\{k^2;(0,\chi_{0}(w),0)\}$; 
if $A=0$, and $C\not=0$, we have $l_3=0$, $l_2=1$, that gives
\[
\left\{
\begin{array}{l}
(2k+k^2)\sqrt{\pi h}w B-\frac{ik}{\sqrt{2}}w C=x\sqrt{\pi h} w B,\\
\sqrt{2}{i \sqrt{\pi h}} B+(k^2+k+1) C=xC,
\end{array}
\right.
\]
with eigenvalues $k^2+k$, and $(k+1)^2$. Otherwise, assuming $ABC\not=0$, $l_1=l_3-1$, $l_2=l_3+1$, and we find

\[
\left\{
\begin{array}{l}
((2l_3+1)k+k^2-k-x)A-i\sqrt{2\pi h l_3} C=0,\\
((2l_3+1)k+k^2+k-x)B-ik\sqrt{\frac{l_3+1}{2\pi h }} C=0,\\
ik\sqrt{\frac{l_3}{2\pi h}} A+i\sqrt{2\pi h(l_3+1)} B+((2l_3+1)k+k^2+1-x) C=0.
\end{array}
\right.
\]

The matrix of the coefficients of this system reads
\[
\left(\begin{array}{ccc} 
(2l_3+1)k+k^2-k-x&0&-i\sqrt{2\pi h l_3} \\
0&(2l_3+1)k+k^2+k-x&-ik\sqrt{\frac{l_3+1}{2\pi h }} \\
ik\sqrt{\frac{l_3}{2\pi h}} &i\sqrt{2\pi h(l_3+1)} &(2l_3+1)k+k^2+1-x
\end{array}
\right).
\]

Setting $x=y+(2l_3+1)k+k^2$, 

\[
\left(\begin{array}{ccc} 
-k-y&0&-i\sqrt{2\pi h l_3} \\
0&k-y&-ik\sqrt{\frac{l_3+1}{2\pi h }} \\
ik\sqrt{\frac{l_3}{2\pi h}} &i\sqrt{2\pi h(l_3+1)} &1-y
\end{array}
\right).
\]

The characteristic equation is
\[
y(y^2-y-(2l+1)k-k^2)=0,
\]
with solutions
\begin{align*}
y&=0,&y&=\frac{1}{2}\pm\frac{1}{2}\sqrt{1+4(2l+1)k+4k^2},
\end{align*}
that give respectively
\begin{align*}
x&=(2l+1)k+k^2,&x&=\left(\sqrt{k(2l+1)+k^2+\frac{1}{4}}\pm \frac{1}{2}\right)^2
\end{align*}
with $l=0,1,2,\dots$. Note that when $l=0$ we find the eigenvalue $k+k^2$, with the same eigenvector as found in the particular case $A=0$ considered above, thus we do not list this eigenvalue separately. This completes the determination of the spectrum of the Hodge Laplace operator.

\begin{prop} The spectrum of the Hodge Laplace operator $\Delta^{(q)}$ on $H$ localised at the representation $\rho_h$ is as follows ($k=2\pi h\in \R-\{0\}$):

\begin{align*}
\Sp d\rho_h(\Delta^{(0)})=&\{(2 m+1)|k|+k^2\}_{m=0}^\infty,\\
\Sp d\rho_h(\Delta^{(1)})
=&\{k^2, (|k|+1)^2\}\cup \{(2 m+1)|k|+k^2\}_{m=0}^\infty\\
&\cup \left\{\left(\sqrt{|k|(2m+1)+k^2+\frac{1}{4}}\pm \frac{1}{2}\right)^2\right\}_{m=0}^\infty.
\end{align*}

Moreover, $\Sp d\rho_h(\Delta^{(2)})=\Sp d\rho_h(\Delta^{(1)})$, and $\Sp d\rho_h(\Delta^{(3)})=\Sp d\rho_h(\Delta^{(0)})$. Each eigenvalue has multiplicity one.
\end{prop}

\subsection{The heat operator}

The localised heat operator in degree $q$ is
\[
d\rho_h \e^{-t\Delta^{(0)}q}=\FF \e^{-t\Delta^{(q)}} \FF^{-1}(h).
\]

This is a trace class operator, and 
\begin{align*}
\Tr d\rho_h \e^{-t\Delta^{(0)}}=&\sum_{m=0}^\infty \e^{-t((2m+1)k+k^2)},\\
\Tr d\rho_h \e^{-t\Delta^{(1)}}=&\e^{-tk^2}+\e^{-t(k+1)^2}+\sum_{m=0}^\infty \e^{-t((2m+1)k+k^2)}\\
&+\sum_{m=0}^\infty \e^{-t\left(\sqrt{k(2m+1)+k^2+\frac{1}{4}}+ \frac{1}{2}\right)^2}+\sum_{m=0}^\infty \e^{-t\left(\sqrt{k(2m+1)+k^2+\frac{1}{4}}- \frac{1}{2}\right)^2}.
\end{align*}

\subsection{The localised analytic torsion}

According to the definition, the analytic torsion of the Heisenberg group $H$ localised at the representation $\rho_h$, is 
$\TF(h)=\tf'(0;H,h)$, where the localised analytic torsion zeta function is 
\[
\tf(s;H,h)=\sum_{q=0}^3 (-1)^q q \frac{d}{ds}\left. \zeta(s,d\rho_h(\Delta^{(q)}))\right|_{s=0},
\]
and
\[
\zeta(s,d\rho_h\Delta^{(q)})=\sum_{\la(h)\in \Sp_+ d\rho_h(\Delta^{(q)})} \la(h)^{-s}.
\]

The last function is clearly defined by the uniformly convergent series when $\Re(s)>1$, and otherwise by its analytic extension. Note that, by direct substitution, cancellations appear in the function $\tf$ that reduces to
\[
\tf(s;H,h)=k^{-2s}+(1+k)^{-2s}-2t_0(s)+t_1(s),
\]
where 
\begin{align*}
t_0(s)=&\sum_{m=0}^\infty  ((2 m+1)k+k^2)^{-s},\\
t_1(s)=&\sum_{m=0}^\infty\left(\sqrt{k(2m+1)+k^2+\frac{1}{4}}+ \frac{1}{2}\right)^{-2s}
+\sum_{m=0}^\infty\left(\sqrt{k(2m+1)+k^2+\frac{1}{4}}-\frac{1}{2}\right)^{-2s}.
\end{align*}

In order to study the analytic extension of $\tf$ and to compute the localised analytic torsion, we proceed as follows. 
First, the case of $t_0(s)=\zeta(s,\Delta^{(0)},\pi_h)$ is quite easy. Indeed,
\begin{align*}
t_0(s)=&k^{-s}\zeta_H(s,k) - (2\la)^{-s}\zeta_H(s,k/2),
\end{align*}
where $\zeta_H$is the Hurwitz zeta function. 
It follows that the analytic extension of $t_0(s)$ has a unique simple pole at $s=1$,  it is regular at $s=0$, and  we compute
\begin{align*}
t_0(0)=&-\frac{k}{2},\\
t_0'(0)=&\frac{1}{2}k\log k+\left(\frac{1}{2}-\frac{k}{2}\right)\log 2+\log\Gamma(k)-\log\Gamma(k/2).
\end{align*}

Next, we consider $t_1$. This requires a bit more work. With
\[
a_m=\sqrt{k(2m+1)+k^2+\frac{1}{4}},
\]
and $b=k+\frac{1}{4k}$, we define
\[
z(s)=\sum_{m=0}^\infty a_m^{-2s}=\sum_{m=0}^\infty \left(k(2m+1)+k^2+\frac{1}{4}\right)^{-s}
=k^{-s}\sum_{m=0}^\infty (2m+1+b)^{-s},
\]
and
\[
\zeta_\pm(s)=\sum_{m=0}^\infty \left( a_m\pm \frac{1}{2}\right)^{-2s}=\sum_{m=0}^\infty \left(\sqrt{k(2m+1)+k^2+\frac{1}{4}}\pm \frac{1}{2}\right)^{-2s}.
\]

Proceeding as above, we write: 
\[
z(s)=k^{-s}\sum_{n=0}^\infty (n+b)^{-s}-k^{-s}\sum_{n=0}^\infty (2n+b)^{-s}
=k^{-s}\zeta_H(s,b)-(2k)^{-s}\zeta_H(s,b/2).
\]

It follows that the analytic extension of $z$ has a unique pole at $s=1$ with residuum
\begin{align*}
\Ru_{s=1} z(s)&=\frac{1}{2k},\\
\Rz_{s=1} z(s)&=\frac{1}{k}\left(\frac{1}{2}\psi(b/2)-\psi(b)\right),
\end{align*}
and
\begin{align*}
z(0)&=-\frac{b}{2}=-\frac{k}{2}-\frac{1}{8k},\\
z'(0)&=\log \frac{\Gamma(b)}{\Gamma(b/2)}+\frac{1}{2}(1-b)\log 2+\frac{b}{2}\log k.
\end{align*}

We use this information as follows. Expanding (for  $\Re(s)$ large)
\begin{align*}
\zeta_\pm(s)=&\sum_{m=0}^\infty a_m^{-2s}
\sum_{j=0}^\infty \binom{-2s}{j}(\pm 1)^j 2^{-j}a_m^{-j},
\end{align*}
we have
\begin{align*}
t_1(s)=&\zeta_+(s)+\zeta_-(s)\\
=&2\sum_{m=0}^\infty a_m^{-2s}
-\frac{s}{2} \sum_{m=0}^\infty a_m^{-2s-2}
+2\sum_{j=2}^\infty \binom{-2s}{2j}2^{-2j}\sum_{m=0}^\infty a_m^{-2s-2j}\\
=&2z(s)+\frac{1}{2}s(1+2s)z(s+1)
+2\sum_{j=2}^\infty \binom{-2s}{2j}2^{-2j}z(2s+2j).
\end{align*}

Note that this means that the possible poles of (the analytic continuation of) $t_1(s)$ are simple and are located at $s=1$ and at the negative half integers $s=-\frac{3}{2}, -\frac{5}{2}, \dots$.

Near $s=0$
\begin{align*}
\frac{1}{2}s(1+2s)z(s+1)=\frac{1}{2}s(1+2s)\left(r_0+\frac{r_1}{s}+O(s^2)\right)=\frac{r_1}{2}+\frac{1}{2}(r_0+2r_1)s+O(s^2),
\end{align*}
where $r_k$ denote the residues, and therefore
\begin{align*}
t_1(s)
=&2z(s)+\frac{1}{2}\Ru_{s=1}z(s)+\frac{1}{2}\left(\Rz_{s=1} z(s)+2\Ru_{s=1} z(s)\right)s+
O(s^2)\\
&+2\sum_{j=2}^\infty \binom{-2s}{2j}2^{-2j}a_m^{-2s-2j}.
\end{align*}

This gives 
\[
t_1(0)=2z(0)+\frac{1}{2}\Ru_{s=1}z(s)=-k.
\]

In order to compute the derivative, note that
\[
\left.\frac{d}{ds}\binom{-2s}{2j}\right|_{s=0}=\frac{1}{j},
\]
then
\begin{align*}
t_1'(0)
=&2z'(0)+\frac{1}{2}\left(\Rz_{s=1} z(s)+2\Ru_{s=1} z(s)\right)
+2\sum_{j=2}^\infty \frac{1}{j}\sum_{m=0}^\infty  2^{-2j} a_m^{-2j}.
\end{align*}

We deal with  the last term as follows. Observing that:
\[
\sum_{j=2}^\infty \frac{1}{j}\left(\frac{1}{2a_m}\right)^{2j}
=\sum_{j=1}^\infty \frac{1}{j}\left(\frac{1}{2a_m}\right)^{2j}-\left(\frac{1}{2a_m}\right)^{2}
=-\log \left(1-\frac{1}{4a^2_m}\right)\e^{\frac{1}{4a^2_m}},
\]
we compute
\begin{align*}
2\sum_{j=2}^\infty \frac{1}{k}\sum_{m=0}^\infty  2^{-2j}a_m^{-2j}
=&-2\log \prod_{m=0}^\infty \left(1-\frac{1/4k}{(2m+1)+k+\frac{1}{4k}}\right)\e^{\frac{1/4k}{(2m+1)+k+\frac{1}{4k}}}\\
&+\frac{1}{2k}\sum_{m=0}^\infty (2m+1)^{-1}-\frac{1}{2k}\sum_{m=0}^\infty (2m+1)^{-1}\\
=&-2\log \prod_{m=0}^\infty \left(1-\frac{1/4k}{(2m+1)+b}\right)\e^{\frac{1/4k}{2m+1}}\\
&-\frac{1}{2k}\sum_{m=0}^\infty ((2m+1)+b)^{-1}+\frac{1}{2k}\sum_{m=0}^\infty (2m+1)^{-1}.
\end{align*}

We tackle the two terms separately. For the second one
\begin{align*}
\frac{1}{2k}&\left(\sum_{m=0}^\infty (2m+1)^{-1}-\sum_{m=0}^\infty ((2m+1)+b)^{-1}\right)\\
&=\frac{1}{2k}\left. \left(\zeta_R(s)-2^{-s}\zeta_R(s)-\zeta_H(s,b)+2^{-s}\zeta_H(s,b/2)\right)\right|_{s=1}.
\end{align*}

Near $s=1$, 
\[
\zeta_H(s,b)=-\psi(b)+\frac{1}{s-1},
\]
and therefore ($\psi(1)=\ga$)
\begin{align*}
\frac{1}{2k}\left(\sum_{m=0}^\infty (2m+1)^{-1}-\sum_{m=0}^\infty ((2m+1)+b)^{-1}\right)=\frac{1}{2k}\left(\frac{\ga}{2}+\psi(b)-\frac{1}{2}\psi(b/2)\right).
\end{align*}

For the first term, we recall the definition of the Euler Gamma function,
\begin{align*}
-2\log \prod_{m=0}^\infty &\left(1-\frac{1/4k}{(2m+1)+b}\right)\e^{\frac{1/4k}{2m+1}}\\
=&-2\log \prod_{n=0}^\infty \left(1-\frac{1/(4k)}{ n+b}\right)
\e^{\frac{1/(4k)}{ n}}+2\log \prod_{n=0}^\infty \left(1-\frac{1/(8k) }{n+b/2}\right)\e^{\frac{1/(8k) }{ n}}\\
=&-2\log\frac{\e^{\ga/(4k)}\Gamma(b+1)}{\Gamma (k+1)}
+2\log \frac{\e^{\ga/(8k)}\Gamma(b/2+1)}{\Gamma (k/2+1)}.
\end{align*}

Thus, 
\begin{align*}
t_1'(0)
=&(1-b)\log 2+b\log k+\frac{1}{2k}+2\log\Gamma(k)-2\log\Gamma(k/2).
\end{align*}

We have proved the following results. 

\begin{prop} The analytic torsion zeta function $\tf(s;H,h)$ of the Heisenberg group $H$ localised at the representation $\rho_h$, $0\not=h\in \hat H$, is a regular analytic function of $s$ for all $s$ with a simple poles at $s=1$,  and possible simple poles at $s=-\frac{3}{2},-\frac{5}{2}, \dots$. Near $s=0$, we have the expansion
\[
\tf(s;H,h)=\tf(0;H,h)+\tf'(0;H,h) s+O(s^2),
\]
where
\begin{align*}
\tf(0;H,h)&=2,\\
\tf'(0;H,h)
&=-2\log k-2\log (1+k)+\frac{1}{4k}\log k-\frac{1}{4k}\log 2+\frac{1}{2k}.
\end{align*}

\end{prop}

\begin{corol} The analytic torsion of the three dimensional Heisenberg group $H$ localised at the representation $\rho_h$, $h\not=0$, is 
\[
\TF(H, h)=
-2\log 2\pi h(1+2\pi h)+\frac{1}{8\pi h}\log \pi h+\frac{1}{4\pi h}.
\]
\end{corol}

\subsection{The relative analytic torsion}

By Assumption \ref{ass} (1), (2), and according to equation (\ref{reltor}), the relative analytic torsion of $(H,\Gamma)$ is (recall that $\tf(s;H,h)=\tf(s;H,-h)$)
\beq
\label{reltorH}
\TF_\Gamma(H)=\left. \frac{d}{ds}\tf_\Gamma(s;H)\right|_{s=0}+2\int_0^\de \TF(H,h)|h| dh.
\eeq
where the relative analytic torsion zeta function (equation (\ref{reltorfun}) is
\beq
\label{reltorfunH}
\tf_\Gamma(s;H)=2\int_\de^\infty \tf(s;H,h) |h| dh.
\eeq

We would like to find a geometric interpretation of this invariant as in Section \ref{ab} for the abelian case. Take $\de=1$,  then,  for large $\Re(s)$, 
by uniform convergence of the series of function  for $\al\in (0,1)$ when $\Re (s)$ is large, this gives
\begin{align*}
\tf_\Gamma(s;H)=&\int_{-\infty}^{-1} \tf(s;H,h) | h|dh+\int_1^\infty \tf(s;H,h) h dh\\
=&\int_{0}^{1} \sum_{n=-\infty}^{-2}\tf(s;H,n+\al) |n+\al|d\al+\int_{0}^{1} \sum_{n=1}^{\infty}\tf(s;H,n+\al) (n+\al)d\al.
\end{align*}

After some simplifications, we find
\beq\label{top}
\begin{aligned}
\tf_\Gamma(s;H)=&\int_{0}^{1} \sum_{n\in\Z} \tf(s;H,n+\al) |n| d\al-\int_0^1  \tf(s;H,1-\al)d\al\\
&+\int_{0}^{1} \sum_{n=1}^{\infty}\tf(s;H,n+\al) \al d\al-\int_{0}^{1} \sum_{n=-\infty}^{-2}\tf(s;H,n+\al) \al d\al.
\end{aligned}
\eeq

We would like to identify  the integrands in the previous formula with the analytic torsion of some smooth Riemannian manifold. Unfortunately, we are only able to partially accomplish this purpose. This is the subject of the next sections. 

\subsection{Some quotients of the Heisenberg group: $H_{\rm red}$ and $N$} 

Following the line indicated in the abelian case, we would like to rewrite the integrands appearing in equation (\ref{top}) as the analytic torsion of some Riemmanian manifold. Unfortunately, this does not work so nicely: the spectral invariant appearing are not the analytic torsion but some less natural ones. In order to proceed we need first some quotient spaces, and second to identify the spectrum of the Hodge Laplace operator on these quotient spaces. 

We start by describing the quotient spaces. The first is the reduced Heinsenberg group, $H_{\rm red}=Z\bsl H$, that is the quotient space of $H$ by the action of the centre of $\Gamma$, i.e.  the subgroup $Z=\{0\}\times\{0\}\times \Z$, see for example \cite[pg. 23]{Fol}. This is a complete smooth Riemannian manifold, with the quotient metric denoted by $g_{H_{\rm red}}$,  universal covering space $H$, and fundamental group $M=\pi_1(H_{\rm red})=\Z$. 

The second quotient is $N=\T^3$ obtained by taking the quotient of $H$ by the left action of $\Z^3$. It is clear that $\T^3$   is a compact smooth Riemannian manifold, with the quotient metric $g_N$, universal covering space $H$ and fundamental group $\pi_1(N)=\Z^3$. Recalling that the quotient space $M=\Gamma\bsl H$ is a compact smooth manifold isomorphic to a circle bundle over the torus $\T^2$, we see that locally $M$ and $N$ are homeomorphic. This means that they share the same Hodge Laplace operator. Indeed, we will prove this  in Section \ref{sprtor} below.

Accomplished the geometric first step, we pass to analysis: we would like to identify the spectrum of the Hodge Laplace operator on $H_{\rm red}$ and on $N$ and to use it to define some spectral invariants (on $N$ precisely the analytic torsion).  In order to identify the spectrum we proceed  adapting the approach of Folland, Auslander and Tolimieri for the compact Heisenberg group  \cite[pg. 68]{Fol} \cite{Fol2} \cite[Section 1]{AT} and \cite{Bre}. 

\subsection{The spectrum of the Hodge Laplace operator and a spectral invariant on $H_{\rm red}$}
\label{spe}

Since $H_{\rm red}$ is not compact, in order to develop spectral analysis we need to work with square integrable forms. In particular, since $H_{\rm red}$ is completed, we work with the unique self adjoint extension in the space of the square integrable forms of the restriction of the formal Hodge Laplace operator on the space of smooth forms with compact support \cite{Gaf}. 

Consider the representation $\pi$ of $H$ on $L^2(H_{\rm red})$ determined by  right translation:
\[
\pi (g)(f) (Zx)=f(Zxg),
\]
and  the $\pi$ invariant  subspaces of $L^2(H_{\rm red})$, $n\in \Z$,
\[
\H_n=\{f\in L^2(H_{\rm red})~|~\pi(0,0,t)f=\e^{2\pi i n t}f\}.
\]

According to the Stone Von Neumann Theorem, the restriction of $\pi_n$ of $\pi$ to $\H_n$ is a direct sum of $\rho_{n}$. This proves that the eigenvalues of the Hodge Laplace operator on $H_{\rm red}$ coincide with  those of the Hodge Laplace operator on $H$, with $h=n\not=0$ (this is known by \cite{MPR}). Next, we consider multiplicity. Let see in some details the case of functions when $h=n=1$. We define the function
\begin{align*}
\Phi_1:&L^2(\R)\to \HH_1\subseteq L^2(H_{\rm red}),\\
\Phi_1:&f\mapsto \Phi_1(f),
\end{align*}
where
\begin{align*}
\Phi_1(f)(p,q,t)&=\e^{2\pi i t}\int_{-\infty}^\infty f(p+s) \e^{2\pi i q s} ds=\e^{2\pi i t}\e^{-2\pi i qp}\int_\R f(v) \e^{2\pi i qv} dv\\
&=\e^{2\pi i t}\e^{-2\pi i qp}\FF (f)(q)=\e^{2\pi i t}\FF(f(\_-p))(q).
\end{align*}

Recall that, as observed above, we are thinking to function in $\SS$, and to the usual completion to square integrable function. Also note that $\Phi_1$ is obviously  $Z$ invariant.
It is clear that $\Phi_1$ preserves the $L^2$ norms, and is therefore and isometry of $L^2(\R)$ into $L^2(\H_1)$.

We show that this map intertwines $\rho_1$ restricted to $\HH_1$ with $\pi$. For compute on one side
\[
\rho_1(a,b,c)(f)(x)=\e^{2\pi i c}\e^{2\pi i bx} f(x+a),
\]
and
\[
\Phi_1\rho_1(a,b,c)(f)(p,q,t)=\e^{2\pi i t}\e^{2\pi i c} \int f(p+s+a) \e^{2\pi i b(p+s)} \e^{2\pi i qs} ds,
\] 
and on the other
\[
\Phi_1(f)(p,q,t)=\e^{2\pi i t}\int_{-\infty}^\infty  f(p+s) \e^{2\pi i q s} ds,
\]
and
\[
\pi(a,b,c) \Phi_1(f)(p,q,t)=\e^{2\pi i t}\e^{2\pi i c} \e^{2\pi i pb}\int f(p+s+a) \e^{2\pi i (q+b)s}  ds.
\]

To conclude we need to verify that $\Phi_1$ is onto. For take $f\in\HH_1$, then we write
\[
f(p,q,t)=\e^{2\pi i  t} g(p,q), 
\]
and let
\[
F(x)=\int_\R \e^{2\pi i (p-x)t} g(p,t) dt=\FF_2(g(p,p-\_))(x)=\FF^{-1}_2(g(p,\_-p))(x).
\]

Then,
\[
\Phi_1(F)(p,q,t)=\e^{2\pi i  t} \FF(\FF^{-1}_2(g(p,\_-p))(\_-p))(q)=\e^{2\pi i  t} g(p,q)=f(p,q,t).
\]


This construction extends to $h=n$ as follows. Define
\begin{align*}
\Phi_n:&L^2(\R)\to \HH_n\subseteq L^2(H_{\rm red}),\\
\Phi_n:&f\mapsto \Phi_n(f),
\end{align*}
where
\[
\Phi_n(f)(p,q,t)=\e^{2\pi i n t}\int_{-\infty}^\infty  f(n(p+s)) \e^{2\pi i  n q s} ds.
\]

Also $\Phi_n$ is $Z$ invariant. 
We show that this map intertwines $\pi$ restricted to $\HH_n$ with $\rho_n$. For compute on one side
\[
\rho_n(a,b,c)(f)(x)=\e^{2\pi i n c}\e^{2\pi i bx} f(x+n a),
\]
and
\[
\Phi_n\rho_n(a,b,c)(f)(p,q,t)=\e^{2\pi i nt}\e^{2\pi i nc} \int f(n(p+s)+na) \e^{2\pi i n b(p+s)} \e^{2\pi i n qs} ds,
\] 
and on the other
\[
\Phi_n(f)(p,q,t)=\e^{2\pi i n t}\int_{-\infty}^\infty  f(n(p+s)) \e^{2\pi i n q s} ds,
\]
and
\[
\pi(a,b,c) \Phi_n(f)(p,q,t)=\e^{2\pi in t}\e^{2\pi in c} \e^{2\pi in pb}\int f(n(p+a+s)) \e^{2\pi i n(q+b)s}  ds.
\]

To conclude we need to verify that $\Phi_n$ is onto. For take $f\in\HH_n$, then we write
\[
f(p,q,t)=\e^{2\pi i n t} g(p,q), 
\]
and expand $g$ in its Fourier transform
\[
f(p,q,t)=\e^{2\pi i n t} \int_{-\infty}^\infty \int_{-\infty}^\infty \hat g(\hat p, \hat q) \e^{2\pi i \hat p p} \e^{2\pi i \hat q q} d\hat p d\hat q.
\]

Define the function $f\in L^2(\R)$:
\[
F(n(x+\hat q))=n \int \hat g(\hat p , n \hat q) \e^{2\pi i \hat p x}d\hat p,
\]
then we verify that $\Phi_n(F)=f$, for compute
\begin{align*}
\Phi_n(F)(p,q,t)&=\e^{2\pi i n t}  \int_{-\infty}^\infty F(n(p+\hat q)) \e^{2\pi i n q \hat q} d\hat q\\
&=\e^{2\pi i n t}  \int_{-\infty}^\infty \int_{-\infty}^\infty \hat g(\hat p , n \hat q) \e^{2\pi i \hat p p}d\hat p  \e^{2\pi i n q \hat q} n d\hat q\\
&=\e^{2\pi i n t}  g(p,q).
\end{align*}

This proves that the eigenvalues of the Hodge Laplace operator on $H_{\rm red}$ all have multiplicity one. 

The last point is to twist the coefficients. For, consider the  finite irreducible representation of $Z$: $\chi_{\al}(0,0,n)=\e^{2\pi i \al n}$, and the induced bundle $E_\al$ over $H_{\rm red}$.  The smooth sections of this bundle may be identified with the function on $H_{\rm red}$ satisfying the conditions
\[
f((0,0,1)(x,y,t))=\e^{2\pi i \al }f(x,y,t).
\]

It follows that the space of the square integrable sections of $H_{\rm red}$ with values in $E_\al$, $L^2(H_{\rm red}, E_\al)$ decomposes in the $\pi$ invariant subspaces
\[
\H_{n+\al}=\{f\in L^2(H_{\rm red}, E_\al)~|~\pi(0,0,t)f=\e^{2\pi i (n+\al) t}f\}.
\]

We may repeat the previous construction and show that $\pi$ restricted to $\HH_{n+\al}$ is equivalent to $\rho_{n+\al}$, and conclude the determination of the spectrum and the prove of the following proposition.

\begin{prop} \label{spectrumHred} The spectrum of the Hodge Laplace operator $\Delta_\al^{(q)}$ on forms on $H_{\rm red}$ with values in $E_\al$, $0<\al<1$,   is as follows:

\begin{align*}
\Sp \Delta_\al^{(0)}=&\{(2 m+1)|n+\al|+(n+\al)^2\}_{m=0,n=-\infty}^\infty,\\
\Sp \Delta_\al^{(1)}
=&\{(n+\al)^2, (|n+\al|+1)^2\}\cup \{(2 m+1)|n+\al|+(n+\al)^2\}_{m=0, n=-\infty}^\infty\\
&\cup \left\{\left(\sqrt{|n+\al|(2m+1)+(n+\al)^2+\frac{1}{4}}\pm \frac{1}{2}\right)^2\right\}_{m=0, n=-\infty}^\infty.
\end{align*}

Moreover, $\Sp \Delta_\al^{(2)}=\Sp \Delta_\al^{(1)}$, and $\Sp \Delta_\al^{(3)}=\Sp \Delta_\al^{(0)}$. Each eigenvalue has multiplicity one.
\end{prop}

Then, a direct verification gives the following result.

\begin{lem}\label{lem1} The spectrum $\Sp \Delta^{(q)}$ of the Hodge Laplace operator $\Delta^{(q)}$ on forms on $H_{\rm red}$ with values in $E_\al$, $0<\al<1$, is a sequence of spectral type of  genus 2. In particular, the associated zeta functions have analytic expansion regular at $s=0$. 
\end{lem}

In order to proceed with our interpretation of the second line of equation (\ref{top}), we need to introduce a suitable spectral invariant on $H_{\rm red}$. This invariant "measures" the spectral asymmetry of the Fourier group decomposition of $L^2(H_{\rm red}, E_\al)$ into the subspaces $\H_{n+\al}$, and is defined as follow. Let $\la^{(q)}_{m,n}(\al)$ denote the eigenvalue of the Hodge Laplace operator $\Delta^{(q)}_\al$ on $H_{\rm red}$  with indices $n$ and $m$ given in Proposition \ref{spectrumHred}. Consider the function of the complex variable $s$, defined for $\Re(s)$ large by the series
\[
e(s; H_{\rm red},\al)=\sum_{m\in\N, n\in \Z} \sgn(n) \la_{m,n}^{-s},
\]
and by analytic extension elsewhere. By Lemma \ref{lem1},  $e(s; H_{\rm red},\al)$ is regular at $s=0$, so we define
\[
E(H_{\rm red},\al)=e'(0; H_{\rm red},\al).
\]

\subsection{The spectrum of the Hodge Laplace operator and the analytic torsion on $N$}
\label{sprtor}

The eigenvalues of the Hodge Laplace operator on $N$ are essentially described in \cite{MPR}. We need also their multiplicity and the the same result with some twist of the coefficients. For we proceed as follows (compare with \cite[3]{Fol2}). 
Consider  the representation $\Rho$ of $H$ on $L^2(N)$,  determined by the right translation,
\begin{align*}
\Rho&: H\to U(L^2(N)),\\
\Rho&: g\mapsto \Rho(g)(f)(\Gamma x)=f(\Gamma xg),
\end{align*}
and the $\Rho$ invariant subspaces of $L^2(N)$
\[
\H_n=\{f\in L^2(N)~|~\Rho(0,0,c)f=\e^{2\pi i n c}f\}.
\]

According to the Stone Von Neumann Theorem, the restriction  $\Rho_n$ of $\Rho$ to $\H_n$ is a direct sum of $\rho_{n}$. This proves that the eigenvalues of the Hodge Laplace operator on $N$ coincide with  those of the Hodge Laplace operator on $H$, with $h=n\not=0$. 

Next, we consider multiplicity.  We proceed adapting the construction of \cite[ch. 1]{AT}. Define the function

\begin{align*}
\Psi_n:&L^2(\R)\to \HH_1\subseteq L^2(N),\\
\Psi_n:&f\mapsto \Psi_n(f),
\end{align*}
where
\[
\Psi_n(f)(p,q,t)=\e^{2\pi i n t}\sum_{k\in\Z}  \e^{2\pi i    k q}  f(np+k).
\]

Here $(p,q,t)$ is in $H$, but we verify that $\Psi_n(f)$ is $\Gamma$ invariant, and therefore defines a function on $N$ as claimed.


We show that $\Phi_n$ map intertwines $\Rho_n$  with $\rho_n$. For compute on one side
\[
\rho_n(a,b,c)(f)(x)=\e^{2\pi i n c}\e^{2\pi i bx} f(x+n a),
\]
and
\[
\Psi_n\rho_n(a,b,c)(f)(p,q,t)=\e^{2\pi i n (t+c)}\sum_{k\in\Z}  \e^{2\pi i    k q} \e^{2\pi ib (np+k)} f(n(p+a)+k) ,
\] 
and on the other
\[
\Psi_n(f)(p,q,t)=\e^{2\pi i n t}\sum_{k\in\Z}  \e^{2\pi i    k q}  f(np+k),
\]
and
\[
\Rho(a,b,c) \Phi_n(f)(p,q,t)=\e^{2\pi i n (t+c+pb)}\sum_{k\in\Z}  \e^{2\pi i    k (q+b)}  f(n(p+a)+k).
\]

We verify that $\Psi_n$ is unitary:
\begin{align*}
\|\Psi_n(f)\|_{L^2(M)}&=\int_0^1\int_0^1\int_0^1 \left|\e^{2\pi i n t}\sum_{k\in\Z}  \e^{2\pi i    k q}  f(np+k)\right|^2 dp dq dt\\
&=\int_0^1\sum_{k\in\Z}  \left|  f(np+k)\right|^2 dp  \\
&=\frac{1}{n}\sum_{k\in\Z}\int_0^n  \left|  f(p+k)\right|^2 dp  \\
&=\|f\|_{L^2(\R)}.
\end{align*}

The next step is to compute multiplicity. 
Take $f\in\HH_n$, then we write
\[
f(p,q,t)=\e^{2\pi i n t} g(p,q), 
\]
and expand $g$ in its Fourier transform
\[
f(p,q,t)=\e^{2\pi i n t}  \sum_{l,k\in \Z}  g_{l,k} \e^{2\pi i lp } \e^{2\pi i kq} .
\]

Define the functions $F_j\in L^2(\R)$, $1\leq j\leq n$, $0\leq u\leq 1$, $m\in\Z$:
\[
F_j(ju+jm)=\sum_{l\in\Z}  g_{l,jm} \e^{2\pi i \frac{l j u}{n}},
\]
then we verify that $\Phi_n(F_j)=f$, for all $j$. For identifying $ju+jm=np+k$
\begin{align*}
\Phi_n(F_j)(p,q,t)&=\e^{2\pi i n t}\sum_{k\in\Z}  \e^{2\pi i    k q}  F_j(np+k)\\
&=\e^{2\pi i n t}\sum_{k\in\Z}  \e^{2\pi i    k q}  \sum_{l\in\Z}  g_{l,k} \e^{2\pi i lp}\\
\end{align*}

This proves that the representation $\Rho_n=\Rho|_{\H_n}$ is equivalent to the sum of $n$ copies of $\rho_n$, and therefore the eigenvalue with index $n$  of the Hodge Laplace operator on $H_{\rm red}$ has multiplicity $n$.

The last effort is to twist the coefficients. The finite dimensional irreducible representations of $\pi_1(N)$ are the characters: $\chi_{\ga,\be,\al}(l,m,n)=\e^{2\pi i (\ga l+\be m+\al n)}$, $0<\ga,\be,\al<1$. 
The relevant ones are the $\e^{2\pi i \al n}$ (i.e. $\ga=\be=0$), with induced bundle $E_\al$ over $N$.    The smooth sections of this bundle may be identified with the function on $N$ satisfying the conditions
\[
f((0,0,1)(x,y,t))=\e^{2\pi i \al }f(x,y,t).
\]

It follows that the space of the square integrable sections of $N$ with values in $E_\al$, $L^2(N, E_\al)$ decomposes in the $\Rho$ invariant subspaces
\[
\H_{n+\al}=\{f\in L^2(N, E_\al)~|~\Rho(0,0,t)f=\e^{2\pi i (n+\al) t}f\}.
\]

We may repeat the previous construction and show that $\Rho$ restricted to $\HH_{n+\al}$ is equivalent to $\rho_{n+\al}$, and conclude the determination of the spectrum and the prove of the following proposition.

\begin{prop} The spectrum of the Hodge Laplace operator $\Delta_\al^{(q)}$ on forms over $N$ with values in $E_\al$, $0<\al<1$,   is as follows:

\begin{align*}
\Sp \Delta^{(0)}=&\{(2 m+1)|n+\al|+(n+\al)^2\}_{m=0,n=-\infty}^\infty,\\
\Sp \Delta^{(1)}
=&\{(n+\al)^2, (|n+\al|+1)^2\}\cup \{(2 m+1)|n+\al|+(n+\al)^2\}_{m=0, n=-\infty}^\infty\\
&\cup \left\{\left(\sqrt{|n+\al|(2m+1)+(n+\al)^2+\frac{1}{4}}\pm \frac{1}{2}\right)^2\right\}_{m=0, n=-\infty}^\infty.
\end{align*}

Moreover, $\Sp \Delta^{(2)}=\Sp \Delta^{(1)}$, and $\Sp \Delta^{(3)}=\Sp \Delta^{(0)}$. Each eigenvalue with index $n$ has multiplicity $n$.
\end{prop}

\subsection{Relative analytic torsion and analytic torsion}

We may use the results in the last two sections to deal with the  terms in formula (\ref{top}) as in the abelian case. First, rewrite the last equation as

\begin{align*}
\tf_\Gamma(s;H)=&\int_{0}^{1} \sum_{n\in\Z} \tf(s;H,n+\al) |n| d\al-2\int_0^1  \tf(s;H,\al)d\al\\
&+\int_{0}^{1} \sum_{n=0}^{\infty}\tf(s;H,n+\al) \al d\al-\int_{0}^{1} \sum_{n=-\infty}^{-1}\tf(s;H,n+\al) \al d\al.
\end{align*}

Next, from one side observe that  we have the equivalence
\[
\sum_{n\in \Z}\tf(s;H,n+\al) |n|=t(s;N,\al),
\]
where $t(s;N,\al)$ is the analytic torsion zeta function of $N$ in the representation $\chi_{0,0,\al}$, and from the other that 
\[
\sum_{n=0}^{\infty}\tf(s;H,n+\al) - \sum_{n=-\infty}^{-1}\tf(s;H,n+\al) =e(s;H_{\rm red},\al),
\]
where $e(s;(H_{\rm red},g_{H_{\rm red}}),\al)$ is the invariant introduced at the end of Section \ref{spe}.

Whence,
\begin{align*}
\tf_\Gamma(s;H)=&\int_0^1 t(s;N,\al)  d\al- 2\int_0^1 \tf(s;H,\al) \al d\al+\int_0^1 e(s;H_{\rm red},\al)\al d\al.
\end{align*}

Proceeding as in Section \ref{ab}, and according to equation (\ref{reltorH}), we have the following result.

\begin{prop}
\begin{align*}
\TF_\Gamma(H)
=&\int_0^1 T(N,\al) \al d\al+\int_0^1 E(H_{\rm red},\al)  d\al.
\end{align*}
\end{prop}

\appendix

\bibliographystyle{plain}

\end{document}